\documentclass[11pt]{article}

\usepackage{comment}
\usepackage[utf8]{inputenc}
\usepackage[english]{babel}
\usepackage{amsmath,amsthm,amsfonts}
\usepackage{graphicx}
\usepackage{bm}
\usepackage{tikz}
\usepackage{multirow}
\usepackage[makeroom]{cancel}
\usepackage{subcaption}

\newtheorem{theorem}{Theorem}[section]
\newtheorem{proposition}[theorem]{Proposition}
\newtheorem{lemma}[theorem]{Lemma}

\newtheorem{assumption}{Assumption}[section]
\graphicspath{ {./figs/}}

\numberwithin{equation}{section}

\renewcommand{\div}{\mathop{\rm div}\nolimits}

\newcommand{\macrogrid}[1]{{\cal T}_{#1}}
\newcommand{\hiergrid}[1]{S_{#1}}

\newcommand{\wc}{\rightharpoonup}

\newcommand{\xoe}{{x\over\epsilon}}
\newcommand{\ep}{\epsilon}

\newcommand{\todo}[1]{{\color{red}{#1}}}

\newcommand{\dx}{ \mathrm{d}x}
\newcommand{\dy}{ \mathrm{d}y}
\newcommand{\dt}{ \mathrm{d}t}

\newcommand{\beq}{\begin{equation}}
\newcommand{\eeq}{\end{equation}}
\newcommand{\beqas}{\begin{eqnarray*}}
\newcommand{\eeqas}{\end{eqnarray*}}

\title{\bf Hierarchical multiscale finite element method for multi-continuum media }
\author{
Jun Sur Richard Park,\\
Department of Mathematics,\\
Texas A\&M University, College Station, TX 77843 \\[10pt]
Viet Ha Hoang,\\
Division of Mathematical Sciences,\\
School of Physical and Mathematical Sciences,\\
 Nanyang Technological University, Singapore 637371 
}
\begin{document}
\maketitle
\begin{abstract}

 Simulation in media with multiple continua where each continuum interacts with every other is often challenging due to multiple scales and high contrast. One needs some types of model reduction. One of the approaches is multi-continuum technique, where every process in each continuum is modeled separately and an interaction term is added. 
Direct numerical simulation in multi scale media is usually not practicable. For this reason, one constructs the corresponding homogenized equations.
Computing the effective coefficients of the homogenized equations can be expensive because one needs to solve local cell problems for a large number of macroscopic points. 
The paper develops a hierarchical approach for solving cell problems at a dense network of macroscopic points with an essentially optimal computation cost.
The method employs the fact that
neighboring representative volume elements (RVEs) share
similar features; and effective properties of the neighboring 
RVEs are close to each other.
The hierarchical approach reduces computation cost by
using different levels of  resolution for cell problems at different macroscopic points. Solutions of the cell problems at macroscopic points where approximation spaces with a higher level of resolution are used are employed to correct the solutions at nearby macroscopic points that are computed by approximation spaces with a lower level of resolution.
The method requires a hierarchy of macroscopic grid points and a corresponding nested approximation spaces with different levels of resolution. Each level of macroscopic points is assigned to an approximation finite element (FE) space which is used to solve the cell problems at the macroscopic points in that level.
We prove rigorously that this hierarchical method achieves the same level of accuracy as that of the full solve where cell problems at every macroscopic point are solved using the FE spaces with the highest level of resolution, but at the essentially optimal computation cost. 
Numerical implementation that computes effective permeabilities of a two scale multicontinuum system via the numerical solutions of the cell problems supports the analytical results.  Finally, 
we prove the homogenization convergence for our multiscale multi-continuum system.
\end{abstract}

\section{Introduction}




Media with multiple continua where each continuum interacts with other continua often entail multiple scales and high contrast. For example, fractured media can have multiscale and high contrast due to complex material properties and geography of fractures. Therefore, numerical simulations in this type of media can be expensive and require model reductions. This can be achieved by computing effective properties in each coarse block using the solutions of local representative volume element (RVE) problems. {In multi-continuum approaches \cite{barenblatt1960basic,warren1963behavior,kazemi1976numerical,wu1988multiple,pruess1982fluid}, equations for each continuum are written separately with so called interaction terms. Therefore, one has to deal with a system of equations and compute the effective properties from it.} 

There have been several methods to solve multiscale equations without computing effective properties and establishing
homogenized equations. The multiscale finite element method (MsFEM) \cite{eh09} solves local cell problems in coarse blocks with fine mesh to obtain basis functions that capture small-scale information. The generalized multiscale finite element method (GMsFEM) 
\cite{chung2016adaptive, egh12, cho2017generalized} follows the outline of MsFEM but adds some degrees of freedom in each coarse block by building the snapshot spaces and solving local spectral problems in the spaces. 
Using GMsFEM for multi-continuum models is considered in \cite{chung2017coupling}.  
Although these methods have several advantages, they do not take into account the  periodicity or local periodicity of the structures of the media of interest and sometimes introduce high computation cost when they have very small fine mesh sizes. In this paper, we establish an efficient algorithm for obtaining homogenized equations utilizing microscopic periodicity for multiscale multicontiuum systems with optimal computation complexity.

The multi-continuum system depends on small 
scales. We study the problem via the homogenization approach. 
For locally periodic multiscale problems such as those considered in this paper, the homogenized problem can be found by the two scale asymptotic expansion \cite{bensoussan1978asymptotic, bakhvalov1989homogenisation, jikov2012homogenization}.
Constructing homogenized equations requires
solving local cell problems in microscale representative domain. 
Computation cost for solving these cell problems can be very high since 
 we have to solve different local cell problems at many macroscopic points. 

In this work we use two-scale asymptotic expansion (\cite{bensoussan1978asymptotic}, \cite{bakhvalov1989homogenisation}) 
to derive the homogenized equations of a two scale multi-continuum system.
Effective coefficients of the homogenized equation are established via the solutions of cell problems.
Since the coefficients depend on both macro- and micro-scale variables, 
a different set of cell equations needs to be solved at each macroscopic point. The number of equations to be solved is thus very large. Solving them using the same small mesh at every macroscopic point is extremely expensive. We develop in this paper the  hierarchical approach to solve the cell problems for the multi-continuum system for a large number of macroscopic points, using an optimal number of degrees of freedom, without sacrificing the accuracy. The method is developed for two scale elliptic problems in \cite{brown2013efficient} (see also \cite{brown2017hierarchical}). It solves cell problems for a dense hierarchical network of macroscopc points with different levels of resolution. The problems at those points belonging to a lower level in the hierarchy are solved with a higher level of accuracy. For the solution at a macroscopic point in a higher level in the hierarchy which are obtained with a lower level of accuracy, we use solutions at nearby macroscopic points that are solved with a higher level of accuracy to correct the error. We show that this hierarchical FE approach obtains the same level of accuracy at every macroscopic point as that obtained when every cell problem is solved with the highest level of resolution (we will refer to this as the full reference solve below), but uses only an essentially optimal number of degrees of freedom that is equal to that required to solve only one cell problem at the finest level of resolution (apart from a possible logarithmic factor).


The paper is organized as follows. In the next section, we 
set up the multiscale multi-continuum system; and we derive the homogenized equations from two scale asymptotic expansion. In Section 3, we outline the hierarchical finite element algorithm for solving the cell problems at a dense network of macroscopic points. We give a rigorous error estimates that show the algorithm has the equivalent accuracy as the full reference solve, at essentially optimal computation cost. In Section 4, we present numerical examples that
verify the theoretical results. We compute the effective permeabilities using the hierarchical solve and the full solve. We find that the effective permeabilities obtained from these two approaches are essentially equal, with a very small relative error between each other.
Finally, in Section 5, we rigorously prove the homogenization convergence for the two scale multi-continuum system.
The paper ends with the conclusions in Section 6.

Throughout the paper, by $\nabla$, we denote the gradient with respect to $x$ of a function that depends only on the spatial variable $x$ and the temporal variable $t$,
and by $\nabla_x$, we denote the partial gradient with respect to $x$ of a function that depends on $x$, $t$ and also other variables. 
Repeated indices indicate summation. The notation $\#$ denotes spaces of periodic functions. 
\section{Problem formulation}
 \subsection{Homogenization of multi-continuum systems}
In multi-continuum approaches, equations for each continuum are written separately. 
We denote by $u_i$ the solution for $i$th continuum. In the general case where each continuum interacts with every other continuum, we have the following system of equations introduced in \cite{chung2017coupling}
\begin{equation*}
\label{eq:multi_cont}
{\mathcal C}_{ii} ^\epsilon(x){\partial u_i^\epsilon(t,x) \over \partial t}=\text{div}(\kappa^\epsilon_i (x) \nabla u_i^\epsilon(t,x)) + Q_i^\epsilon(u_1^\epsilon(t,x),...,u_N^\epsilon(t,x)),\enspace\enspace \textrm{in} \enspace \Omega
\end{equation*}
where $\Omega \subset \mathbb{R}^d$ is a domain ($d = 2, 3$), $\kappa_i^\epsilon$ are the multiscale  permeability and $C_{ii}^\epsilon$ are the multiscale  porosities, and the functions $Q_i^\epsilon$ of $(u_1,...,u_N)$ are  exchange terms (see \cite{barenblatt1960basic,warren1963behavior,kazemi1976numerical,wu1988multiple,pruess1982fluid}) that describe the interaction of continua; $\epsilon$ represents the microscopic scale of the local variation. 
 If each continuum only interacts with the background $u_1^\ep$, we have
\[
{\mathcal C}_{ii}^\epsilon(x){\partial u_i^\epsilon(t,x) \over \partial t}=\text{div}(\kappa_i^\epsilon(x) \nabla u_i^\epsilon(t,x)) + Q_i^\epsilon(u_1^\epsilon(t,x), u_i^\epsilon(t,x)) + q,\ \ i>1,
\]
where $q$ is the source term. 
In this paper, we consider a two-continuum system. Let $Y$ be the unit cube in $\mathbb{R}^d$. Let ${\mathcal C}_{ii}(x,y)$, $\kappa_i(x,y)$ ($i=1,2$) be continuous functions on $\Omega\times Y$ which are $Y$-periodic with respect to $y$ and $q$ be a function in $L^2(\Omega)$. We assume further that there is a constant $c>0$ such that for all $x\in\Omega,y\in Y$
\beq
{\mathcal C}_{ii}(x,y)\ge c,\ \kappa_i(x,y)\ge c,\ Q(x,y)\ge c.
\label{eq:coercivity}
\eeq
We define the twoscale coefficients as
\[
{\mathcal C}_{ii}^\epsilon(x)={\mathcal C}_{ii}(x,{x\over\epsilon}),\ \ \kappa_i^\epsilon(x)=\kappa_i(x,{x\over\epsilon}),\ \ Q^\ep(x)=Q(x,{x\over\epsilon}).
\]
We consider in this paper the case where the interaction terms are scaled as $O(1/\epsilon^2)$; this case has the most interesting cell problems in the form of a coupled system. We consider
the multiscale multi-continuum system

\begin{equation}
\label{eq:main3'}
 \begin{split}
{\mathcal C}_{11}^\epsilon(x){\partial u_1^\epsilon(t,x)  \over \partial t}=\text{div}(\kappa_1^\epsilon(x)\nabla u_1^\epsilon(t,x)) + {1 \over \epsilon^2}Q^\epsilon(x)(u_2^\epsilon(t,x)-u_1^\epsilon(t,x)) + q,\\
{\mathcal C}_{22}^\epsilon(x){\partial u_2^\epsilon(t,x) \over \partial t}=\text{div}(\kappa_2^\epsilon(x) \nabla u_2^\epsilon(t,x)) + {1 \over \epsilon^2}Q^\epsilon(x)(u_1^\epsilon(t,x)-u_2^\epsilon(t,x)) + q,
 \end{split}
 \end{equation} 
 with the Dirichlet boundary condition $u_1^\epsilon(t,x)=u_2^\epsilon(t,x)=0$ for $x\in\partial \Omega$, and with the initial condition $u_1^\epsilon(0,x)=g_1$, $u_2^\epsilon(0,x)=g_2$ where $g_1$ and $g_2$ are in $L^2(\Omega)$.
We consider the following two scale asymptotic expansion of $u_1^\epsilon$ and $u_2^\epsilon$.
\begin{equation*}
\label{eq:main6}
 \begin{split}
u^{\epsilon}_1(t,x) = u_{10} (t,x,{x\over\epsilon})+ \epsilon u_{11}(t,x,{x\over\epsilon}) + \cdots,\ \ 
u^{\epsilon}_2(t,x) = u_{20} (t,x,\xoe)+ \epsilon u_{21}(t,x,\xoe)+ \cdots,
 \end{split}
 \end{equation*}
 where the functions $u_{1i}(t,x,y)$ and $u_{2i}(t,x,y)$ are periodic with respect to $y$. 
 Performing the two scale asymptotic expansion, from (\ref{eq:main3'}) we obtain
 \begin{equation}
  \label{eq:main7}
 \begin{split}
{\mathcal C}_{11}&{\partial (u_{10} + \epsilon u_{11} + \cdots) \over \partial t}\\ &= (\div_x + {1 \over \epsilon} \div_y)(\kappa_1 (\nabla_x + {1 \over \epsilon} \nabla_y)(u_{10} + \epsilon u_{11}+ \cdots))
+ {1 \over \epsilon^2} Q(u_{20}+ \epsilon u_{21} - u_{10}- \epsilon u_{11}+ \cdots) + q, \\
{\mathcal C}_{22}&{\partial (u_{20} + \epsilon u_{21} + \cdots) \over \partial t}\\ &= (\div_x + {1 \over \epsilon} \div_y)(\kappa_2 (\nabla_x + {1 \over \epsilon} \nabla_y)(u_{20} + \epsilon u_{21}+ \cdots))
+ {1 \over \epsilon^2} Q(u_{10}+ \epsilon u_{11} - u_{20}- \epsilon u_{21}+ \cdots) + q, 
 \end{split}
 \end{equation}
For the  $O(\epsilon^{-2})$ terms, we obtain,
  \begin{equation*}
  \label{eq:main8}
 \begin{split}
&\div_y(\kappa_1(x,y)\nabla_y u_{10}(t,x,y))+Q(x,y)(u_{20}(t,x,y) - u_{10}(t,x,y)) = 0\\
&\div_y(\kappa_2(x,y)\nabla_y u_{20}(t,x,y))+Q(x,y)(u_{10}(t,x,y) - u_{20}(t,x,y)) = 0.
 \end{split}
 \end{equation*}
From this, we have
  \begin{equation*}
  \label{eq:main9}
 \begin{split}
-\int_Y \kappa_1\nabla_y u_{10} \cdot \nabla_y u_{10} \mathrm{d}y + \int_Y Q(u_{20}-u_{10})u_{10} \mathrm{d}y = 0\\
-\int_Y \kappa_2\nabla_y u_{20} \cdot \nabla_y u_{20} \mathrm{d}y  + \int_Y Q(u_{10}-u_{20})u_{20} \mathrm{d}y  = 0
 \end{split}
 \end{equation*}
Adding these two equations, we obtain
   \begin{equation*}
  \label{eq:main10}
 \begin{split}
\int_Y \kappa_1\nabla_y u_{10} \cdot \nabla_y u_{10}\mathrm{d}y  + \int_Y \kappa_2\nabla_y u_{20} \cdot \nabla_y u_{20}\mathrm{d}y 
 +\int_Y Q(u_{20}-u_{10})^2\mathrm{d}y = 0.\\
 \end{split}
 \end{equation*}
This implies $\nabla_y u_{10} = 0$, $\nabla_y u_{20} = 0$. i.e. $u_{10}$ and $u_{20}$ are independent of $y$, and $u_{10}(t,x) = u_{20}(t,x) $ as $Q(x,y)>c>0\ \forall\,x\in\Omega,y\in Y$. For the $O(\epsilon^{-1})$ terms in (\ref{eq:main7}), we have,
\begin{equation*}
\label{eq:main11}
\begin{split}
&\div_x(\kappa_1\nabla_y u_{10}) + \div_y(\kappa_1\nabla_x u_{10})+\div_y (\kappa_1\nabla_y u_{11})+Q(u_{21}-u_{11}) = 0\\
&\div_x(\kappa_2\nabla_y u_{20}) + \div_y(\kappa_2\nabla_x u_{20})+\div_y (\kappa_2\nabla_y u_{21})+Q(u_{11}-u_{21}) = 0.
\end{split}
\end{equation*}
Since $u_{10}$ and $u_{20}$ are independent of $y$, we have
\begin{equation*}
\label{eq:main12}
\begin{split}
\div_y (\kappa_1\nabla_y u_{11})+Q(u_{21}-u_{11}) = -\div_y(\kappa_1\nabla u_{10})\\
\div_y (\kappa_2\nabla_y u_{21})+Q(u_{11}-u_{21}) = -\div_y(\kappa_2\nabla u_{20})
\end{split}
\end{equation*}
Thus $u_{11} = {\partial u_{10} \over \partial x_i} N^i_1 $ and $u_{21} = {\partial u_{20} \over \partial x_i} N^i_2$ 
where $N_1^i(x,\cdot)\in H^1_\#(Y)/\mathbb{R}$, and $N_2^i(x,\cdot)\in H^1_\#(Y)/\mathbb{R}$ are solutions of the cell problem
\begin{equation}
  \label{eq:cell}
 \begin{split}
   \text{div}_y(\kappa_1(x,y) (e^i+\nabla_y N^i_{1})) +
   Q(x,y) ( N^i_{2}- N^i_{1}) =0\\
   \text{div}_y(\kappa_2(x,y) (e^i+\nabla_y N^i_{2})) +
   Q(x,y) ( N^i_{1}- N^i_{2}) =0
 \end{split}
\end{equation}
where $e^i$ is the $i$th unit vector in the standard basis of ${\mathbb R}^d$.
For the  $O(\epsilon^0)$ terms in (\ref{eq:main7}), integrating over $Y$, one has
\begin{equation*}
 \label{eq:main14}
 \begin{split}
\int_Y {\mathcal C}_{11} {\partial u_{10} \over \partial t}\mathrm{d}y  = \int_Y \div_x(\kappa_1 \nabla u_{10}) \mathrm{d}y 
+ \int_Y \div_x (\kappa_1 \nabla_y u_{11})\mathrm{d}y  +\int_Y Q(u_{22}- u_{12})\mathrm{d}y + \int_Y q \mathrm{d}y \\
\int_Y {\mathcal C}_{22} {\partial u_{20} \over \partial t}\mathrm{d}y 
 = \int_Y \div_x(\kappa_2 \nabla u_{20}) \mathrm{d}y + \int_Y \div_x (\kappa_2 \nabla_y u_{21})\mathrm{d}y  +\int_Y Q(u_{12}- u_{22})\mathrm{d}y 
 + \int_Y q\mathrm{d}y 
 \end{split}
\end{equation*}
Adding these two equations, one obtains the homogenized equation
\begin{equation}
 \label{eq:main15}
 \begin{split}
\left(\int_Y {\mathcal C}_{11} \mathrm{d}y )+\int_Y {\mathcal C}_{22}  \mathrm{d}y \right){\partial u_{0} \over \partial t} 
= \div(\kappa^*_1 \nabla u_0) + \div(\kappa^*_2 \nabla u_0)
+\int_Y 2q \mathrm{d}y  \qquad     \text{in} \enspace \Omega 
\end{split}
\end{equation}
where $u_0 = u_{10} = u_{20}$ and the $x$-dependent permeabilities are defined as
\begin{equation}
 \label{eq:main16}
 \begin{split}
\kappa^*_{1ij}(x) =  \int_Y  \kappa_1 (x,y)(\delta_{ij} + {\partial  N^j_1(x,y)\over \partial y_i}) dy,\ \ 
\kappa^*_{2ij}(x) =  \int_Y  \kappa_2 (x,y)(\delta_{ij} + {\partial  N^j_2(x,y)\over \partial y_i}) dy
\end{split}
\end{equation}
We will show later that the homogenized matrices $\kappa^*_{1ij}(x)$ and $\kappa^*_{2ij}(x)$ are positive definite.
We will also show that the initial condition for $u_0$ is
\beq
u_0(0,x)={\langle C_{11}\rangle g_1(x)+\langle C_{22}\rangle g_2(x)\over\langle C_{11}\rangle+\langle C_{22}\rangle}
\label{eq:initcond}
\eeq
where $\langle C_{ii}\rangle=\int_YC_{ii}(y)dy$ for $i=1,2$. Equation \eqref{eq:main15} together with initial condtion \eqref{eq:initcond} has a unique solution (see, e.g., \cite{wlokapartial}). 

\subsection{Uniqueness of solution to the cell problem}
We  write the system (\ref{eq:cell}) in the variational form as
\begin{equation}
\label{eq:main17}
\begin{split}
\int_Y \kappa_1(x,y) \nabla_y N^i_1(x,y) \cdot \nabla_y \phi_1(y)\mathrm{d}y  
- \int_Y Q(x,y)(N^i_2 - N^i_1)\phi_1(y)\mathrm{d}y  = - \int_Y \kappa_1(x,y) e^i \cdot \nabla_y \phi_1(y)\mathrm{d}y \\
\int_Y \kappa_2(x,y) \nabla_y N^i_2(x,y) \cdot \nabla_y \phi_2(y) \mathrm{d}y - \int_Y Q(x,y)(N^i_1 - N^i_2)\phi_2(y)\mathrm{d}y 
 = - \int_Y \kappa_2(x,y) e^i \cdot \nabla_y \phi_2(y)\mathrm{d}y 
\end{split}
\end{equation}
where $\phi_1,\phi_2 \in H^1_{\#}(Y)$.
Let $W$ be the space $H^1_\#(Y)\times H^1_\#(Y)/(c,c), c\in \mathbb{R}$. The space $W$ is equipped with the norm
\[
|||(\phi_1,\phi_2)|||=\|\nabla_y\phi_1\|_{L^2(Y)}+\|\nabla_y\phi_2\|_{L^2(Y)}+\|\phi_1-\phi_2\|_{L^2(Y)}.
\]
 For $x\in\Omega$, we define the bilinear form $B(x;\cdot,\cdot):W\times W\to\mathbb{R}$ as
\beqas
B(x;(\phi_1,\phi_2),(\psi_1,\psi_2))&=&\int_Y\kappa_1(x,y)\nabla_y\phi_1(y)\cdot\nabla_y\psi_1(y)dy+\int_Y\kappa_2(x,y)\nabla_y\phi_2(y)\cdot\nabla_y\psi_2(y)dy\\
&&+\int_YQ(x,y)(\phi_1(x,y)-\phi_2(x,y))(\psi_1(x,y)-\psi_2(x,y))dy
\eeqas
for $(\phi_1,\phi_2)\in W$ and $(\psi_1,\psi_2)\in W$. 
From \eqref{eq:coercivity}, we deduce that the bilinear form $B$ is uniformly coercive and bounded with respect to $x\in\Omega$, i.e. there are constants $c_1>0$ and $c_2>0$ such that
\[
B(x;(\phi_1,\phi_2),(\phi_1,\phi_2))\ge c_1|||(\phi_1,\phi_2)|||^2,\ \mbox{and} \ B(x;(\phi_1,\phi_2),(\psi_1,\psi_2))\le c_2|||(\phi_1,\phi_2)|||\cdot|||(\psi_1,\psi_2)|||
\]
for all $(\phi_1,\phi_2)\in W$ and $(\psi_1,\psi_2)\in W$. 
Adding the two equations in (\ref{eq:main17}), we obtain
\begin{equation*}
\label{eq:main18}
\begin{split}
 B(x;(N_1^i,N_2^i)(\phi_1,\phi_2))
= - \int_Y \kappa_1(x,y) e^i \cdot \nabla_y \phi_1(y) \mathrm{d}y - \int_Y \kappa_2(x,y) e^i \cdot \nabla_y \phi_2(y)\mathrm{d}y. 
\end{split}
\end{equation*}

\begin{theorem}
Problem (\ref{eq:main17}) has a unique solution $(N^i_1, N^i_2) \in W$.
\end{theorem}
\begin{proof}
The conclusion follows from the boundedness and coerciveness of the bilinear form $B$ and the Lax-Milgram lemma. 
\end{proof}

\section{Hierarchical finite element algorithm}

%

Computing effective coefficients $\kappa_i^*(x)$
requires the solutions of the cell problems (\ref{eq:cell}) at many macroscopic points
which can be very expensive.
We develop in this section the hierarchical FE method which computes the solution of the cell problems at a dense network of macroscopic points using only an essentially optimal number of degrees of freedom which is equal to that for solving one cell problem (apart from a multiplying logarithmic factor).
We assume that the coefficients are sufficiently smooth with respect to  the macroscopic variable $x$. 
We make the following assumption.
\begin{assumption}
\label{Lipschitz}
There is a constant $C > 0$ such that for all $x$, $x' \in \Omega$,
\begin{equation*}
\label{eq:main23''}
\begin{split}
\|\kappa_1(x,\cdot) - \kappa_1(x',\cdot)\|_{L^\infty(Y)}\le C|x-x'|,\ \|\kappa_2(x,\cdot) - \kappa_2(x',\cdot)\|_{L^\infty(Y)}\le C|x-x'|,\\
\mbox{and}\ \|Q(x,\cdot) - Q(x',\cdot)\|_{L^\infty(Y)} \leq C|x-x'|
\end{split}
\end{equation*}
\end{assumption}
%

\subsection{ Overview of hierarchical algorithm}


We develop an efficient hierarchical finite element algorithm to solve the coupled cell problem (\ref{eq:cell}) numerically and to approximate the effective properties $\kappa_i^*(x)$ in \eqref{eq:main16} for a dense network of macroscopic points $x\in \Omega$.
We follow the algorithm introduced in \cite{brown2013efficient}. 

We outline the algorithm as follows.

\textbf{Step 1 : Build nested finite element spaces.}
We employ Gelerkin FE to obtain an approximation of the solution $(N^i_1, N^i_2)\in W$ of (\ref{eq:cell}) for each macroscopic point $x \in \Omega$ using FE spaces of different levels of resolution.
We assume that there exists a hierarchy of FE spaces
${\cal V}_{0} \subset {\cal V}_{1}\subset\cdots\subset {\cal V}_{L} \subset H^1_{\#}(Y)$  where the integer index $L$ denotes the resolution level.
We assume further the following approximation properties: for $w\in H^2_\#(Y)$, 
\beq
\inf_{\phi\in{\mathcal V}_{L-l}}\|\nabla_y(w-\phi)\|_{L^2(Y)}+2^{L-l}\|w-\phi\|_{L^2(Y)}\le C2^{-L+l}\|w\|_{H^2(Y)},
\label{eq:FEapprox}
\eeq
where
the constant $C$ is independent of $L$ and $l$.  


\textbf{Step 2 : Build a hierarchy of macrogrids.}
We solve the cell equations at different macroscopic points $x \in \Omega$ with different levels of accuracy. We use the solutions solved with a higher  accuracy level to correct the solutions obtained with a lower accuracy level. We achieve this by solving the cell problems at different macroscopic points using different FE spaces in the hierarchy in Step 1.
This can be done by constructing a hierarchy of macro-grid points.
We construct a nested macro-grid,
 $\macrogrid{0}\subset \macrogrid{1}\subset\cdots\subset \macrogrid{L} \subset \Omega$
as follows.  
First, we build an initial grid $\macrogrid{0}$ with a proper grid spacing $H$, the maximal distance between neighboring nodes.
We then inductively construct $\macrogrid{l}$, a refinement of
$\macrogrid{l-1}$,
with grid spacing $H 2^{-l}$. 
Then, we define the hierarchy of macro-grids, 
$\{\hiergrid{0},\hiergrid{1},\cdots,\hiergrid{L}\}$
as $\hiergrid{0}=\macrogrid{0}$, $\hiergrid{1}=\macrogrid{1}\backslash \hiergrid{0}$, 
and for each $l > 1$, we have
 \begin{equation*}
\label{eq:main24}
\begin{split}
 \hiergrid{l}=\macrogrid{l}\Big\backslash \bigg({\bigcup_{k<l}\hiergrid{k}}\bigg)
 \end{split}
\end{equation*}
 We call the nodes in the lowest level grid $\hiergrid{0}$ the anchor points.
In this way, we obtain a dense hierarchy of the macro-grids. That is, 
each point $x \in \hiergrid{l}$ has at least one point from one of the previous levels,
 $x' \in \bigcup_{k<l} \hiergrid{k}$ such that dist($x$, $x'$) $<$ $O(H 2^{-l})$. 
 Figures \ref{fig:macrogrid} and \ref{fig:hiergrid} show an example of 3-level hierarchy of macrogrids
 $\macrogrid{l}$, $\hiergrid{l}$, $l = 1,2,3$, constructed in $\Omega = [0, 1]^2$.\\
 \begin{figure}[t!]
 \centering
  \begin{subfigure}[b]{0.255\linewidth}
    \includegraphics[width=\linewidth]{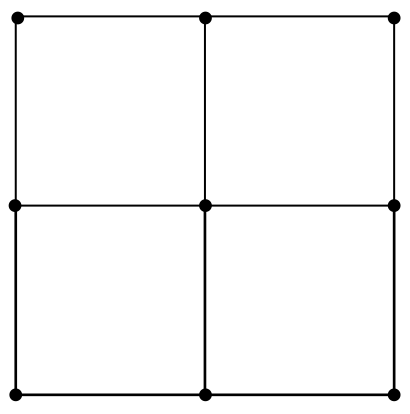}
    \caption{$\macrogrid{0}$}
  \end{subfigure}
  ~~~~
  \begin{subfigure}[b]{0.255\linewidth}
    \includegraphics[width=\linewidth]{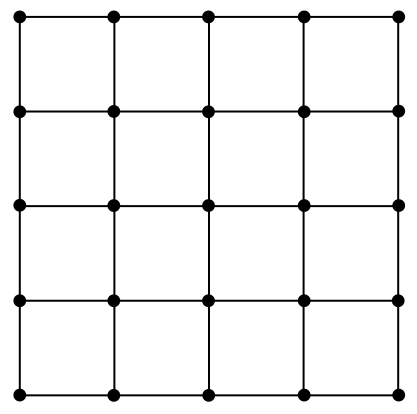}
    \caption{$\macrogrid{1}$}
  \end{subfigure}
  ~~~~
   \begin{subfigure}[b]{0.255\linewidth}
    \includegraphics[width=\linewidth]{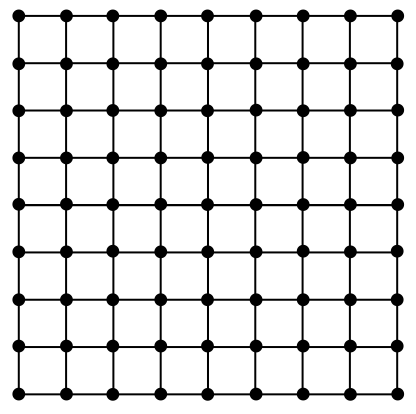}
    \caption{$\macrogrid{2}$}
  \end{subfigure}
  \caption{3-level nested macrogrids}
  \label{fig:macrogrid}

  \begin{subfigure}[b]{0.255\linewidth}
    \includegraphics[width=\linewidth]{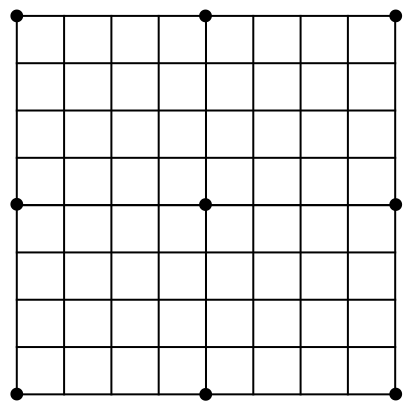}
    \caption{$\hiergrid{0}$}
  \end{subfigure}
  ~~~~
  \begin{subfigure}[b]{0.255\linewidth}
    \includegraphics[width=\linewidth]{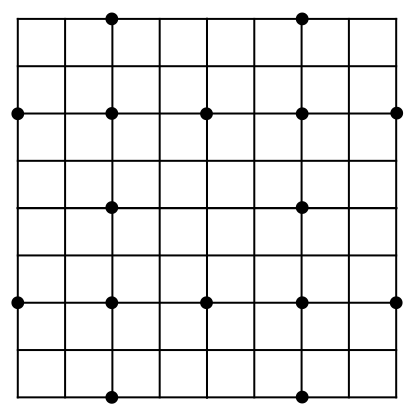}
    \caption{$\hiergrid{1}$}
  \end{subfigure}
  ~~~~
   \begin{subfigure}[b]{0.255\linewidth}
    \includegraphics[width=\linewidth]{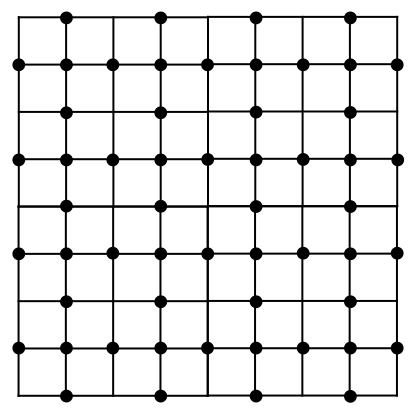}
    \caption{$\hiergrid{2}$}
  \end{subfigure}
  \caption{3-level hierarchy of macrogrids}
  \label{fig:hiergrid}
\end{figure}
%
%

 \textbf{Step 3 : Calculating the correction term.}
 We relate the nested FE spaces and the hierarchy of macrogrids for our algorithm. We first solve the cell problems at  anchor points using the standard Galerkin FE with FE space $\mathcal{V}_L$. 
 That is, for the points in the coarsest macro-grid $\hiergrid{0}$,
 we solve the  cell problems with the finest mesh.
 More precisely, we find $\bar{N^i_1}(x,\cdot), \bar{N^i_2}(x,\cdot)  \in \mathcal{V}_L$, such that
\begin{equation*}
\label{eq:main24'}
\begin{split}
B(x;(\bar N_1^i,\bar N_2^i),(\phi_1,\phi_2))=- \int_Y \kappa_1(x,y) e^i \cdot \nabla_y \phi_1(y)\mathrm{d}y-\int_Y \kappa_2(x,y) e^i \cdot \nabla_y \phi_2(y)\mathrm{d}y
\end{split}
\end{equation*}
for all $\phi_1,\phi_2 \in \mathcal{V}_L$. 
Proceeding inductively, for $x \in \hiergrid{l}$ ($l=1,\cdots,L$), we choose the points  $\{x_{1},x_{2},\cdots,x_{n}\} \in (\bigcup_{l'<l}\hiergrid{l'})$ so that the distance between $x$ and each point in $\{x_{1},x_{2},\cdots,x_{n}\} $ is 
$O(H2^{-l})$.
This is possible from the assumption for the hierarchy of macroscopic points above. 
We define the $l$-th macro-grid interpolation by
\begin{equation*}
\label{eq:interpolation}
 I_{l}^{x}(N^i_k)=\sum_{j=1}^{n}c_{j} N^i_k(x_{j},\cdot),
 \end{equation*}
 where the coefficients $c_{j}$ satisfy $\sum_{j=1}^{n}c_{j}=1$ ($k=1,2$). We refer to the $l$-th macro-grid interpolation of Galerkin approximations
 as $I_{l}^{x}(\bar{N^i_k})=\sum_{j=1}^{n}c_{j}\bar{N^i_k}(x_{j},\cdot)$.
  We solve the following problem: 
 Find ${\bar{N^i_1}}^c(x,\cdot), {\Bar {N^i_2}}^c (x,\cdot) \in {\cal V}_{L-l}$
  such as
\begin{equation}
\begin{split}
&B(x;({\bar{N^i_1}}^c, {\bar{N^i_2}}^c),(\phi_1,\phi_2))\\
&=  
-\sum_{j=1}^n c_j  \int_Y (\kappa_1(x,y) -  \kappa_1(x_j,y))\nabla_y \bar{N^i_1}(x_j,y) \cdot \nabla_y \phi_1(y) \mathrm{d}y-\sum_{j=1}^n c_j \int_Y (\kappa_1(x,y) - \kappa_1(x_j,y)) e^i \cdot \nabla_y \phi_1(y) \mathrm{d}y\\
&-\sum_{j=1}^n c_j  \int_Y (\kappa_2(x,y) -  \kappa_2(x_j,y))\nabla_y \bar{N^i_2}(x_j,y) \cdot \nabla_y \phi_2(y)\mathrm{d}y-\sum_{j=1}^n c_j \int_Y (\kappa_2(x,y) - \kappa_2(x_j,y)) e^i \cdot \nabla_y \phi_2(y)\mathrm{d}y\\
&\qquad\qquad+\sum_{j=1}^{n}c_{j} \int_Y (Q(x_j,y)-Q(x,y)) (\bar{N^i_1}(x_{j},y)- \bar{N^i_2}(x_{j},y))(\phi_1(y)- \phi_2(y)) \mathrm{d}y,
\end{split}
\label{eq:bN1icN2ic}
\end{equation}
for all $\phi_1$, $\phi_2$ $\in \mathcal{V}_{L-l}$.
Note that right-hand side data is all known since we have already computed $\{\bar{N^i_k}(x_{j},\cdot)\}_{j=1}^{n}$ inductively using finer mesh spaces at macro-grid points in
$(\bigcup_{l'<l}\hiergrid{l'})$. 
We let
\begin{equation}
\label{eq:main25''}
\bar{N^i_k}(x,\cdot)=\bar{N^i_k}^{c}(x,\cdot)+I_{l}^{x}(\bar{N^i_k}), 
\end{equation}
be the FE approximation for $N^i_k(x,\cdot)$ where $k = 1,2$.
A main goal of this paper is to prove that the approximation (\ref{eq:main25''}) for 
$N^i_k (x,\cdot)$ has the same order of accuracy compared to the approximation we obtain by solving
 (\ref{eq:main17}) using the finest FE space $\mathcal{V}_{L}$ at all macroscopic points.
We also prove that we reduce the computation cost with the approximation (\ref{eq:main25''}) to the optimal level.\\ \\
\textit{Remark.} In the following, for simplicity, we use a simple 1-point interpolation to compute the correction term $(\bar{N_1^i}^c,\bar{N_2^i}^c)$. More precisely, for $x \in \hiergrid{l}$ we choose
$x' \in (\bigcup_{l'<l}\hiergrid{l'})$ such that $ \mathrm{dist} (x, x') < O(H2^{-l})$. We let
\begin{equation*}
\label{eq:main26}
\begin{split}
I^x_l(\bar{N}^i_k) = \bar{N}^i_k(x', \cdot), \ \ k = 1,2
\end{split}
\end{equation*}
be the macro-grid interpolation. The FE approximation is
\begin{equation*}
\label{eq:main27}
\begin{split}
 \bar{N}^i_k(x,  \cdot)=\bar{N^i_k}^c(x,  \cdot) +  \bar{N}^i_k(x',  \cdot), \ \ k = 1,2.
\end{split}
\end{equation*}\\ 
\textit{Remark.} Note that as the level $l$ goes higher, we use coarser FE spaces for the corresponding finer macro grids. This balance guarantees that although we use coarser FE spaces, the FE error is still optimal, but with much less computation cost.

\subsection{Error estimates}
We require that the coefficients $\kappa_i$ and $Q$ satisfy Assumption \ref{Lipschitz} and \eqref{eq:coercivity}. We prove that the hierarchical method achieves the same order of accuracy as the full solve. 
For simplicity, we consider 1-point interpolation for our proof; the proof for the general case is similar.
\begin{lemma}
\label{lemma1}
There exists a positive number $C$ such that
$|||(N_1^i(x,\cdot),N_2^i(x,\cdot))|||\leq C$ for all $x\in\Omega$.
\end{lemma}
\begin{proof}
From \eqref{eq:main17}, we obtain
\begin{equation*}
\label{eq:main27''}
\begin{split}
&B(x;(N_1^i(x,\cdot),N_2^i(x,\cdot)),(N_1^i(x,\cdot),N_2^i(x,\cdot)))\\
&= - \int_Y \kappa_1(x,y) e^i \cdot \nabla_y N^i_1(x,y)\mathrm{d}y - \int_Y \kappa_2(x,y) e^i \cdot \nabla_y N^i_2(x,y) \mathrm{d}y. 
\end{split}
\end{equation*}
Using the uniform coercivity of the bilinear form $B(x;\cdot,\cdot)$ with respect to $x$, we get
\begin{equation*}
\label{eq:main27'''}
\begin{split}
C|||(N_1^i(x,\cdot),N_2^i(x,\cdot))||| \leq (||\nabla_y N^i_1(x,\cdot)||_{L^2(Y)} + ||\nabla_y N^i_2(x,\cdot)||_{L^2(Y)})
\end{split}
\end{equation*}
for $C > 0$. From this we get the conclusion.  
\end{proof}

Let ${N^i_k}^c(x,\cdot)={N^i_k}(x,\cdot)-{N^i_k}(x',\cdot)$. We have that
$({N^i_1}^c(x,\cdot), {N^i_2}^c(x,\cdot))\in W$ satisfies 
\begin{eqnarray}
&&B(x;({N^i_1}^c,{N^i_2}^c),(\phi_1,\phi_2))\nonumber\\
&&=-\int_Y (\kappa_1(x,y) -  \kappa_1(x',y))\nabla_y {N^i_1}(x',y) \cdot \nabla_y \phi_1(y) \mathrm{d}y-\int_Y (\kappa_1(x,y) - \kappa_1(x',y)) e^i \cdot \nabla_y \phi_1(y) \mathrm{d}y\nonumber\\
&&-\int_Y (\kappa_2(x,y) -  \kappa_2(x',y))\nabla_y {N^i_2}(x',y) \cdot \nabla_y \phi_2(y)\mathrm{d}y-\int_Y (\kappa_2(x,y) - \kappa_2(x',y)) e^i \cdot \nabla_y \phi_2(y)\mathrm{d}y\nonumber\\
&&\qquad\qquad+\int_Y (Q(x',y)-Q(x,y)) ({N^i_1}(x',y)- {N^i_2}(x',y))(\phi_1(y)- \phi_2(y)) \mathrm{d}y
\label{eq:N1icN2ic}
\end{eqnarray}
$\forall\,(\phi_1,\phi_2)\in W$. 
\begin{proposition}
\label{prop:1}
There exists $C > 0$ such that 
\[
|||({N^i_1}^c(x,\cdot),{N^i_2}^c(x,\cdot))||| \leq C|x - x'|
\]
for $x\in {\mathcal T}_L$.
\end{proposition}
\begin{proof}
From \eqref{eq:N1icN2ic}, for $(\phi_1,\phi_2)=({N_1^i}^c(x,\cdot),{N_2^i}^c(x,\cdot))\in W$ we have 
\begin{equation*}
\label{eq:main30} 
\begin{split}
&B(x;({N_1^i}^c,{N_2^i}^c),({N_1^i}^c,{N_2^i}^c))\\
&=  
-\int_Y (\kappa_1(x,y) - \kappa_1(x',y))\nabla_y N^i_1(x',y) \cdot \nabla_y {N^{i}_1}^c(x,y)\mathrm{d}y 
-\int_Y (\kappa_1(x,y) - \kappa_1(x',y) e^i \cdot \nabla_y {N^{i}_1}^c(x,y)\mathrm{d}y 
\\&-\int_Y (\kappa_2(x,y) - \kappa_2(x',y))\nabla_y N^i_2(x',y) \cdot \nabla_y {N^{i}_2}^c(x,y)\mathrm{d}y -\int_Y (\kappa_2(x,y) - \kappa_2(x',y)) e^i \cdot \nabla_y {N^{i}_2}^c(x,y)\mathrm{d}y \\
&+\int_Y (Q(x',y)-Q(x,y))(N^i_1(x',y)-N^i_2(x',y))({N^{i}_1}^c(x,y)-{N_2^i}^c(x,y))\mathrm{d}y .
\end{split}
\end{equation*}
As $\nabla_y N^i_1(x',\cdot)$ and $\nabla_y N^i_2(x',\cdot)$ are uniformly bounded in $L^2(Y)$ with respect to $x\in\Omega$ by Lemma \ref{lemma1}.
From Assumption \ref{Lipschitz}
  we have
\begin{equation*}
\label{eq:main31} 
\begin{split}
&|||({N_1^i}^c(x,\cdot),{N_2^i}^c(x,\cdot))|||^2\\
&\qquad\qquad\leq C|x-x'|(||\nabla_y {N^i_1}^c(x,\cdot)||_{L^2(Y)} + ||\nabla_y {N^i_2}^c(x,\cdot)||_{L^2(Y)} + ||{N^i_2}^c(x,\cdot) - {N^i_1}^c(x,\cdot)||_{L^2(Y)}).
\end{split}
\end{equation*}
Thus
\begin{equation}
\label{eq:main32} 
\begin{split}
|||({N_1^i}^c(x,\cdot),{N_2^i}^c(x,\cdot))|||\leq C|x - x'|
\end{split}
\end{equation}
where the constant $C$ is independent of $x$.
\end{proof}
\begin{lemma}
\label{lemma2}
There is a positive constant $C$ such that
$||\Delta_y N^i_1(x,\cdot)||_{L^2(Y)}+ ||\Delta_y N^i_2(x,\cdot)||_{L^2(Y)}  \le C$ for all $x\in\Omega$.
\end{lemma}
\begin{proof}
We rewrite cell problem (\ref{eq:cell}) as 
\begin{equation*}
\label{eq:main32''} 
\begin{split}
\kappa_1\Delta_y N^i_1 + \nabla_y \kappa_1\cdot \nabla_y N^i_1+ \div_y (\kappa_1 e^i) + Q(x,y)(N^i_2-N^i_1) = 0
\end{split}
\end{equation*}
\begin{equation*}
\label{eq:main32'''} 
\begin{split}
\kappa_2\Delta_y N^i_2 + \nabla_y \kappa_2\cdot \nabla_y N^i_2+ \div_y (\kappa_2 e^i) + Q(x,y)(N^i_1-N^i_2) = 0.
\end{split}
\end{equation*}
Rearranging these equations, we have,
\begin{equation*}
\label{eq:main32*} 
\begin{split}
\Delta_y N^i_1 = -\frac{1}{\kappa_1} (\nabla_y \kappa_1\cdot \nabla_y N^i_1+ \div_y (\kappa_1 e^i) + Q(x,y)(N^i_2-N^i_1))
\end{split}
\end{equation*}
\begin{equation*}
\label{eq:main32**} 
\begin{split}
\Delta_y N^i_2 = -\frac{1}{\kappa_2} (\nabla_y \kappa_2\cdot \nabla_y N^i_2+ \div_y (\kappa_2 e^i) + Q(x,y)(N^i_1-N^i_2)).
\end{split}
\end{equation*}
By the uniform boundedness of $|||(N_1^i(x,\cdot),N_2^i(x,\cdot))|||$ with respect to $x$ and Lemma \ref{lemma1}, we deduce that $||\Delta_y N^i_1(x,\cdot)||_{L^2(Y)}$ and $||\Delta_y N^i_2(x,\cdot)||_{L^2(Y)}$ are uniformly bounded for all $x\in\Omega$.
\end{proof}

\begin{lemma}
\label{prop:2}
There exists a positive constant $C$ such that
\begin{equation*}
\label{eq:main32***}
||\Delta_y {N^i_1}^c(x,\cdot)||_{L^2(Y)}\le C|x-x'|,\ \ ||\Delta_y {N^i_2}^c(x,\cdot)||_{L^2(Y)} < C|x-x'| 
\end{equation*}
for all $x\in{\mathcal T}_L$.
\end{lemma}
\begin{proof}
From \eqref{eq:N1icN2ic}, we have
\begin{equation*}
\label{eq:main33}
\begin{split}
\kappa_1(x,y) \Delta_y {N^{i}_1}^c(x,y) &+\nabla_y \kappa_1(x,y)\cdot \nabla_y {N^{i}_1}^c(x,y) = - Q(x,y)({N^{i}_2}^c(x,y) - {N^{i}_1}^c(x,y)) \\  & 
-\nabla_y(\kappa_1(x,y) - \kappa_1(x',y))\cdot\nabla_y N^i_1(x',y) - (\kappa_1(x,y) - \kappa_1(x',y))\Delta_y N^i_1(x',y)
\\&-\div_y(\kappa_1(x,y) - \kappa_1(x',y) e^i)
+(Q(x',y)-Q(x,y))(N^i_2(x',y)-N^i_1(x',y)),
\end{split}
\end{equation*}
\begin{equation*}
\label{eq:main33'}
\begin{split}
\kappa_2(x,y) \Delta_y {N^{i}_2}^c(x,y) &+\nabla_y \kappa_2(x,y) \cdot\nabla_y {N^{i}_2}^c(x,y) = - Q(x,y)({N^{i}_1}^c(x,y) - {N^{i}_2}^c(x,y)) \\  & 
-\nabla_y(\kappa_2(x,y) - \kappa_2(x',y))\cdot\nabla_y N^i_2(x',y) - (\kappa_2(x,y) - \kappa_2(x',y))\Delta_y N^i_2(x',y)
\\&-\div_y(\kappa_2(x,y) - \kappa_2(x',y) e^i)
+(Q(x',y)-Q(x,y))(N^i_1(x',y)-N^i_2(x',y)).
\end{split}
\end{equation*}
Therefore,
\begin{equation*}
\label{eq:main33''}
\begin{split}
\Delta_y {N^{i}_1}^c(x,y) &= \frac{1}{\kappa_1}\{-\nabla_y \kappa_1(x,y)\cdot \nabla_y {N^{i}_1}^c(x,y)  - Q(x,y)({N^{i}_2}^c(x,y) - {N^{i}_1}^c(x,y)) \\  & 
-\nabla_y(\kappa_1(x,y) - \kappa_1(x',y))\cdot\nabla_y N^i_1(x',y) - (\kappa_1(x,y) - \kappa_1(x',y))\Delta_y N^i_1(x',y)
\\&-\div_y(\kappa_1(x,y) - \kappa_1(x',y) e^i)
+(Q(x',y)-Q(x,y))(N^i_2(x',y)-N^i_1(x',y))\},
\end{split}
\end{equation*}
\begin{equation*}
\label{eq:main33'''}
\begin{split}
\Delta_y {N^{i}_2}^c(x,y) &= \frac{1}{\kappa_2}\{- \nabla_y \kappa_2(x,y)\cdot \nabla_y {N^{i}_2}^c(x,y)   - Q(x,y)({N^{i}_1}^c(x,y) - {N^{i}_2}^c(x,y)) \\  & 
-\nabla_y(\kappa_2(x,y) - \kappa_2(x',y))\cdot\nabla_y N^i_2(x',y) - (\kappa_2(x,y) - \kappa_2(x',y))\Delta_y N^i_2(x',y)
\\&-\div_y(\kappa_2(x,y) - \kappa_2(x',y) e^i)
+(Q(x',y)-Q(x,y))(N^i_1(x',y)-N^i_2(x',y))\}.
\end{split}
\end{equation*}
From Lemma \ref{lemma1} and Proposition \ref{prop:1}, we have
\begin{equation*}
\label{eq:main33*}
\begin{split}
||\Delta_y {N^i_1}^c(x,\cdot)||_{L^2(Y)}, ||\Delta_y {N^i_2}^c(x,\cdot)||_{L^2(Y)} < C|x-x'|. \\
\end{split}
\end{equation*}
for some constant $C > 0$.
\end{proof}
We choose $({N_1^i}^c,{N_2^i}^c)\in W$ such that
\[
\int_Y({N_1^i}^c+{N_2^i}^c)dy=0.
\]
We then have
\begin{lemma}
\label{lemma3} 
There is a positive constant $C$ such that 
$||{N^i_1}^c(x,\cdot)||_{L^2(Y)}\le C|x-x'|$ and  $||{N^i_2}^c(x,\cdot)||_{L^2(Y)}\le C|x-x'|$ for all $x\in{\mathcal T}_L$.
\end{lemma}
\begin{proof}
We note that
\begin{equation}
\label{eq:main33**}
\begin{split}
2(||{N^i_1}^c||^2_{L^2(Y)} + ||{N^i_2}^c||^2_{L^2(Y)}) = ||{N^i_1}^c+{N^i_2}^c||^2_{L^2(Y)}+||{N^i_1}^c-{N^i_2}^c||^2_{L^2(Y)}.
\end{split}
\end{equation}
Since $\int_Y ({N^i_1}^c+ {N^i_2}^c) \dy = 0$, by Poincare inequality, and $(\ref{eq:main32})$, the following inequalities hold.
\begin{equation*}
\label{eq:main33***}
\begin{split}
||{N^i_1}^c+{N^i_2}^c||_{L^2(Y)} \leq C||\nabla_y ({N^i_1}^c+{N^i_2}^c)||_{L^2(Y)} \leq C (||\nabla_y{N^i_1}^c||_{L^2(Y)}+||\nabla_y{N^i_2}^c||_{L^2(Y)})  \leq C|x-x'|
\end{split}
\end{equation*}
And then by (\ref{eq:main33**}),
\begin{equation*}
\label{eq:main33****}
\begin{split}
2(||{N^i_1}^c||^2_{L^2(Y)} + ||{N^i_2}^c||^2_{L^2(Y)}) \leq C |x-x'|^2.
\end{split}
\end{equation*}
\end{proof}
\begin{proposition}
\label{prop:3}
There is a constant $C>0$ such that
$||{N^{i}_1}^c||_{H^2(Y)}\le C|x-x'|$ and  $||{N^{i}_2}^c||_{H^2(Y)} \leq C |x-x'|$ for all $x\in{\mathcal T}_L$.
\end{proposition}
\begin{proof}
Let $\omega\subset\mathbb{R}^d$ be a domain such that $Y\subset\omega$. Let $\phi \in {\cal C}^\infty_0(\omega)$ be such that $\phi = 1$ in $Y$. We have
\begin{equation*}
\label{eq:main34}
\begin{split}
\Delta_y(\phi {N^i_1}^c) = \Delta_y \phi {N^i_1}^c + 2 \nabla\phi\cdot \nabla {N^i_1}^c + \phi \Delta_y {N^i_1}^c.
\end{split}
\end{equation*}
Since $\phi {N^i_1}^c = 0$ on $\partial \omega$, applying elliptic regularity, we have
\begin{equation}
\label{eq:main35}
\begin{split}
||{N^i_1}^c||_{H^2(Y)} \leq ||\phi {N^i_1}^c||_{H^2(\omega)} \leq  ||\Delta_y \phi {N^i_1}^c + 2 \nabla_y\phi\cdot \nabla_y {N^i_1}^c + \phi \Delta_y {N^i_1}^c||_{L^2(\omega)}.
\end{split}
\end{equation}
By Proposition \ref{prop:1}, Lemmas \ref{prop:2} and \ref{lemma3},
and the $Y$-periodicity of ${N^i_1}^c$,
\begin{equation*}
\label{eq:main38}
\begin{split}
||{N^i_1}^c(x,\cdot)||_{L^2(\omega)}\le C|x-x'|, ||\nabla_y {N^i_1}^c(x,\cdot)||_{L^2(\omega)}\le C|x-x'|, ||\Delta_y {N^i_1}^c(x,\cdot)||_{L^2(\omega)} \le  C|x-x'|
\end{split}
\end{equation*}
for all $x\in{\mathcal T}_L$. 
Then from \eqref{eq:main35}, 
$
||{N^i_1}^c||_{H^2(Y)}\le C|x-x'|. 
$
Similarly,
 $
 ||{N^i_2}^c||_{H^2(Y)} \le C|x-x'|
 $
for $C > 0$.
\end{proof}
We consider the problem: Find ${\Bar{\Bar{N}}^i_1}^c(x,y)\in{\cal V}_{L-l}$ and ${\Bar{\Bar{N}}^i_2}^c(x,y)\in{\cal V}_{L-l}$ such that
\begin{eqnarray}
&&B(x;({\bar{\bar{N}}^i_1}^c,{\bar{\bar{N}}^i_2}^c),(\phi_1,\phi_2))\nonumber\\
&&=-\int_Y (\kappa_1(x,y) -  \kappa_1(x',y))\nabla_y {N^i_1}(x',y) \cdot \nabla_y \phi_1(y) \mathrm{d}y-\int_Y (\kappa_1(x,y) - \kappa_1(x',y)) e^i \cdot \nabla_y \phi_1(y) \mathrm{d}y\nonumber\\
&&-\int_Y (\kappa_2(x,y) -  \kappa_2(x',y))\nabla_y {N^i_2}(x',y) \cdot \nabla_y \phi_2(y)\mathrm{d}y-\int_Y (\kappa_2(x,y) - \kappa_2(x',y)) e^i \cdot \nabla_y \phi_2(y)\mathrm{d}y\nonumber\\
&&\qquad\qquad+\int_Y (Q(x',y)-Q(x,y)) ({N^i_1}(x',y)- {N^i_2}(x',y))(\phi_1(y)- \phi_2(y)) \mathrm{d}y,
\label{eq:bbN1icN2ic}
\end{eqnarray}
for all $\phi_1\in{\cal V}_{L-l}$ and $\phi_2\in{\cal V}_{L-l}$. This is the FE approximation of \eqref{eq:N1icN2ic}.
We then have the following result.
\begin{lemma} There is a positive constant $C^0$ such that
\label{lem:5}
\begin{equation*}
\label{eq:main44} 
|||({N^i_1}^c(x,\cdot) - {\Bar{\Bar N}^i_1}^c(x,\cdot), {N^i_2}^c(x,\cdot) - {\Bar{\Bar N}^i_2}^c(x,\cdot))|||\leq C^0 2^{-L}.
\end{equation*}
\end{lemma}
\begin{proof}
It follows from Cea's Lemma, Proposition \ref{prop:3} and \eqref{eq:FEapprox} that
\begin{equation*}
\label{eq:main44} 
\begin{split}
|||({N_1^i}^c-{\bar{\bar N}^i_1}^c,{N_2^i}^c-{\bar{\bar N}^i_2}^c)|||\le C2^{-(L-l)}(\|{N_1^i}^c\|_{H^2(Y)}+\|{N_2^i}^c\|_{H^2(Y)})\le C2^{-(L-l)}|x-x'|\le C^02^{-L}.
\end{split}
\end{equation*}
\end{proof}
\begin{proposition}
\label{prop:4}
There is a constant $c_l>0$ which only depends on the level $S_l$ of $x\in{\cal T}_L$ such that
\[
|||(\bar N_1^i(x,\cdot)-N_1^i(x,\cdot), \bar N_2^i(x,\cdot)-N_2^i(x,\cdot))|||\le c_l2^{-L}.
\]
\end{proposition}
\begin{proof}
We will prove the proposition by induction. The conclusion holds for $l = 0$.  
We assume that for all $x'\in S_{l'}$ where $l' \le l-1$. 
\begin{equation}
\label{eq:main55'} 
\begin{split}
|||(\bar{N^{i}_1}(x',\cdot)-N^i_1(x',\cdot),\bar{N^{i}_2}(x',\cdot)-N^i_2(x',\cdot))|||\leq c_{l-1} 2^{-L}.
\end{split}
\end{equation}
From \eqref{eq:bN1icN2ic} and \eqref{eq:bbN1icN2ic}, we have
\beqas
&&B(x;({\bar {N_1^i}}^c(x,\cdot)-{\bar{\bar N}^i_1}^c(x,\cdot),{\bar {N_2^i}}^c(x,\cdot)-{\bar{\bar N}^i_2}^c(x,\cdot)),(\phi_1,\phi_2))\\
&&=-\int_Y (\kappa_1(x,y) - \kappa_1(x',y))\nabla_y (\bar{N^{i}_1}(x',y)-N^i_1(x',y)) \cdot \nabla_y \phi_1(y)\mathrm{d}y\\
&&-\int_Y (\kappa_2(x,y) - \kappa_2(x',y))\nabla_y (\bar{N^{i}_2}(x',y)-N^i_2(x',y)) \cdot \nabla_y\phi_2(y)\mathrm{d}y\\
&&+\int_Y (Q(x',y)-Q(x,y))((\bar{N^{i}_1}(x',y)-\bar{N^{i}_2}(x',y))-(N^i_1(x',y)-N^i_2(x',y))) (\phi_1(y)-\phi_2(y))\mathrm{d}y
\eeqas
for all $\phi_1\in{\cal V}_{L-l}$ and $\phi_2\in{\cal V}_{L-l}$. 
From Assumption \ref{Lipschitz} and the induction hypothesis, we have
\begin{equation}
\label{eq:main56} 
\begin{split}
|||({\bar {N_1^i}}^c(x,\cdot)-{\bar{\bar N}^i_1}^c(x,\cdot),{\bar {N_2^i}}^c(x,\cdot)-{\bar{\bar N}^i_2}^c(x,\cdot))|||\le \gamma c_{l-1}2^{-L-l}.
\end{split}
\end{equation}
where $\gamma > 0$ is independent of $x$ and $l$.
By Lemma \ref{lem:5} and (\ref{eq:main56}), 
\begin{equation}
\label{eq:main57} 
\begin{split}
|||({N_1^i}^c(x,\cdot)-{\bar {N_1^i}}^c(x,\cdot), {N_2^i}^c(x,\cdot)-{\bar {N_2^i}}^c(x,\dot))|||\le |||({N^i_1}^c(x,\cdot) - {\Bar{\Bar{N^i_1}}}^c(x,\cdot),{N^i_2}^c(x,\cdot) - {\Bar{\Bar{N^i_2}}}^c(x,\cdot))|||\\
+|||({\bar{N^i_1}}^c(x,\cdot) - {\Bar{\Bar{N^i_1}}}^c(x,\cdot),{\bar {N^i_2}}^c(x,\cdot) - {\Bar{\Bar{N^i_2}}}^c(x,\cdot))|||\le C^02^{-L}+\gamma c_{l-1} 2^{-L-l}.
\end{split}
\end{equation}
Using $\bar{N}^i_k(x, y)=\bar{N^i_k}^c(x, y)+\bar{N}^i_k(x', y)$, We have
\[
|||(N_1^i(x,\cdot)-\bar N^i_1(x,\cdot), N_2^i(x,\cdot)-\bar N_2^i(x,\cdot))|||\le c_l2^{-L},
\]
where 
\beq
c_l=\gamma c_{l-1}2^{-l}+c_{l-1}+C^0.
\label{eq:cl}
\eeq
\end{proof}
\begin{theorem}
Under Assumption \ref{Lipschitz} and the uniform boundedness of $\kappa_i(x,y)$ and $Q(x,y)$, there is a positive constant $C_*$ which depends only on the functions $\kappa_1$, $\kappa_2$ and $Q$ so that,
\begin{equation}
\label{eq:main55''''} 
|||(N_1^i(x,\cdot)-N_1^i(x,\cdot), N_2^i(x,\cdot)-N_2^i(x,\cdot))|||\le C_*l2^{-L}
%
\end{equation}
for $x\in S_l$. 
\end{theorem}
\begin{proof}
We let $\bar{l}$ be an integer independent of $L$ such that $l2^{-l} < \frac{1}{2\gamma}$ for $l > \bar{l}$. And let
\begin{equation}
\label{eq:main55*} 
\begin{split}
C_{*} = \displaystyle\max \bigg\{\max_{0 \leq l \leq \bar{l}} \Big\{\frac{c_{l}}{l}\Big\}, 2 C^0\bigg\},
\end{split}
\end{equation}
where $C^0$ and $c_{l}$ are the constants in Lemma \ref{lem:5} and Proposition \ref{prop:4}. Now we prove
\begin{equation}
\label{eq:main55**} 
\begin{split}
|||(N^i_1(x,\cdot)-\bar N^{i}_1(x,\cdot), N^i_2(x,\cdot)-\bar N^i_2(x,\cdot))||| \leq C_{*} l 2^{-L}
\end{split}
\end{equation}
by induction. From (\ref{eq:main55*}), this holds for all $l \leq \bar{l}$. 
Suppose that \eqref{eq:main55**} holds for all $l' \leq l$. Then from (\ref{eq:cl}), we obtain
\begin{equation}
\label{eq:main55***} 
\begin{split}
&C_{l} \leq ((l-1) C_{*} + \frac{1}{2\gamma} \gamma C_{*} + \frac{C_{*}}{2} ) = C_{*}l.
\end{split}
\end{equation}
\end{proof}
\begin{theorem}
\label{theoremdegree}
The total number of degrees of freedom required to solve (\ref{eq:main17}) for all points in $\hiergrid{0}, \hiergrid{1}, \cdots, \hiergrid{L} $ is $\mathcal{O}((L+1)2^{dL})$
for the hierarchical solve while it is $\mathcal{O}((2^{dL})^2)$ in the full solve where cell problems are solved with the finest mesh level at all macrogrid points. 
\end{theorem}
\begin{proof}
Since the number of macroscopic points in $\hiergrid{l}$ is $\mathcal{O}(2^{dl})$, and the space $\mathcal{V}_{L-l}$ is of dimension $\mathcal{O}(2^{d(L-l)})$, 
the total number of degrees of freedom for solving (\ref{eq:main17}) for all points in $S_{l}$ is $\mathcal{O}(2^{dl})\mathcal{O}(2^{d(L-l)}) = \mathcal{O}(2^{dL})$. 
Therefore, the total number of degrees of freedom required to solve (\ref{eq:main17}) for all points in $\hiergrid{0}, \hiergrid{1}, \cdots, \hiergrid{L} $ is $\mathcal{O}((L+1)2^{dL})$.
\end{proof}

\section{Numerical example}

In this section, we apply the hierarchical finite element algorithm to a numerical example for computing the effective coefficients of a multiscale multi-continuum system at a dense network of macrogrid points. 
To show the accuracy of the algorithm, we compare the results to the approximations to the effective coefficients obtained from 
solving the cell problems using the finest meshes at all macroscopic points. 
\subsection{Numerical Implementation}
We let $\Omega = [0, 1]^2$ be the macroscopic domain and $Y = [0,1]^2$ be the unit cell. 
We 
consider the locally periodic coefficients
\begin{equation*}
\label{eq:main55'} 
\begin{split}
&\kappa_1(x_1,y_1,y_2) = (2-ax_1)\cos(2\pi y_1)\sin(2\pi y_2)+3 \\
&\kappa_2(x_1,y_1,y_2) = (2-ax_1)\sin(2\pi y_1)\cos(2\pi y_2)+3\\
&Q(x_1,y_1,y_2) = (1+ax_1)\sin(2\pi y_1)\sin (2\pi y_2)+3
\end{split}
\end{equation*}
where the constant $a$ is chosen below.
 We use 4 square meshes in $ [0, 1]^2$ to construct a nested sequence of FE spaces, $\{{{\cal V}_{3-l}} \}^3_{l=0}$ so that the mesh size of
 each space is $h_l = 2^l (2^{-4})$ for $l = 0, 1, 2, 3$.
Since $\kappa_1$, $\kappa_2$ and $Q$ are independent of $x_2$, we only consider 1-dimensional macrogrids in $[0,1]$.
The nested macrogrids $\{ \macrogrid{l}\}^L_{l=0} \subset [0,1]$ and the subsequent macrogrid hierarchy, $\{ \hiergrid{l}\}^3_{l=0}$ are constructed as follows.
We first let $\macrogrid{0} =\hiergrid{0} = \{0, \frac{1}{2},1\}$.
Considering that our macrogrids have grid spacing $H2^{-l}$ for $l = 0, 1, 2, 3$, where $H=\frac{1}{2}$ in this case, we have following hierarchy of macrogrids.
\begin{equation*}
\hiergrid{0} = \{0,\frac{1}{2},1\}, \enspace \hiergrid{1} = \{\frac{1}{4},\frac{3}{4}\},\enspace \hiergrid{2} = \{\frac{1}{8},\frac{3}{8},\frac{5}{8},\frac{7}{8}\}, 
\enspace\hiergrid{3} = \{\frac{1}{16},\frac{3}{16},\frac{5}{16},\frac{7}{16},\frac{9}{16},\frac{11}{16},\frac{13}{16},\frac{15}{16}\}
\end{equation*}
Figure \ref{schematic} indicates how these macrogrids and the approximation spaces are related in numerical implementation.

\begin{figure}
  \centering
  \includegraphics[scale=0.6]{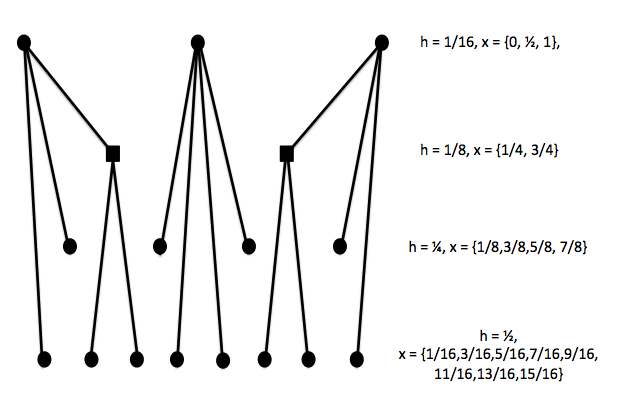}
  \caption{The hierarchy of one dimensional macrogrids and corresponding mesh size of FE spaces. The lines indicates correction relations. The squares indicate the points
  at which the solutions are corrected with the lower level solutions and used once more to correct upper level solutions.}
  \label{schematic}
\end{figure}

We implement the algorithm as follows. For $x' \in \hiergrid{0} = \{0, \frac{1}{2},1\}$, we solve (\ref{eq:main17}) for $\bar{N^i_1}(x', \cdot)$, $\bar{N^i_2}(x',\cdot) \in {\cal V}_3$, for all $\phi_1, \ \phi_2 \in {\cal V}_3$ by the standard Galerkin FEM.
We then use a simple 1-point interpolation to compute the correction terms.
That is, for $x \in \hiergrid{l}$ we choose $x' \in (\bigcup_{k<l}\hiergrid{k})$ 
such that $|x' - x| \leq 2^{-l}$. We let the $l$th macrogrid interpolation be
\begin{equation*}
\label{eq:main46}
\begin{split}
I^x_l(\bar{N}^i_k) = \bar{N}^i_k(x', \cdot), \ (k = 1,2).
\end{split}
\end{equation*} 
We find $\bar{N^{i}_1}^c(x,y)$ and $\bar{N^{i}_2}^c(x,y)$ in $\mathcal{V}_{L-l}$ such that
\begin{equation}
\label{eq:main47'} 
\begin{split}
\int_Y &\kappa_1(x,y) \nabla_y {\bar{N^{i}_1}}^c(x,y) \cdot  \nabla_y \phi_1(y)\mathrm{d}y
-\int_Y Q(x,y)({\bar{N^{i}_2}}^c(x,y) - {\bar{N^{i}_1}}^c(x,y)) \phi_1(y)\mathrm{d}y\\ = & 
-\int_Y \kappa_1(x,y) \nabla_y \bar{N^{i}_1}(x',y) \cdot \nabla_y \phi_1(y)\mathrm{d}y
-\int_Y \kappa_1(x,y)  e^i \cdot \nabla_y \phi_1(y)\mathrm{d}y
\\&+\int_Y Q(x,y)(\bar{N^{i}_2}(x',y)-\bar{N^{i}_1}(x',y)) \phi_1(y)\mathrm{d}y,
\end{split}
\end{equation}
and
\begin{equation}
\label{eq:main47''*} 
\begin{split}
\int_Y &\kappa_2(x,y) \nabla_y {\bar{N^{i}_2}}^c(x,y) \cdot  \nabla_y \phi_2(y) \mathrm{d}y
 -\int_Y Q(x,y)({\bar{N^{i}_1}}^c(x,y) - {\bar{N^{i}_2}}^c(x,y))\phi_2(y)\mathrm{d}y\\ = & 
-\int_Y \kappa_2(x,y) \nabla_y \bar{N^{i}_2}(x',y) \cdot \nabla_y \phi_2(y) \mathrm{d}y
-\int_Y \kappa_2(x,y) e^i \cdot \nabla_y \phi_2(y) \mathrm{d}y
\\&+\int_Y Q(x,y)(\bar{N^{i}_1}(x',y)-\bar{N^{i}_2}(x',y))\phi_2(y) \mathrm{d}y,
\end{split}
\end{equation}
for $\forall \phi_1,\phi_2 \in \mathcal{V}_{L-l}$. 
We let
\begin{equation*}
\label{eq:main48}
\begin{split}
 \bar{N}^i_k(x, \cdot) = \bar{N}^i_k(x', \cdot) + \bar{N^i_k}^c(x, \cdot), \ (k = 1,2)
\end{split}
\end{equation*}
be the approximation to $N_k^i(x,\cdot)$.
We continue  inductively.
For example, for $x = \frac{1}{2} \in \hiergrid{0}$, we compute $\bar{N^i_1}(\frac{1}{2}, \cdot)$, $\bar{N^i_2}(\frac{1}{2},\cdot)$ using the standard Galerkin FEM.
Then for $\frac{3}{8} \in \hiergrid{1}$, we find the correction terms 
$\bar{N^i_1}^c(\frac{3}{8}, \cdot)$, $\bar{N^i_2}^c(\frac{3}{8},\cdot) \in {\cal V}_{L-1}$ that satisfy (\ref{eq:main47'}) and (\ref{eq:main47''*}),
where $x' = \frac{1}{2}$.
And we let the solutions at $x = \frac{3}{8}$ be
\begin{equation*}
\label{eq:main49}
\begin{split}
 \bar{N}^i_k(\frac{3}{8}, y) = \bar{N}^i_k(\frac{1}{2}, y) + \bar{N^i_k}^c(\frac{3}{8}, y), \ (k = 1,2).
\end{split}
\end{equation*}
We continue this procedure based on Figure \ref{schematic}. 

Tables \ref{table:a1} and \ref{table:a01} indicate $\kappa^*_{111}$ and $\kappa^*_{211}$ obtained by both the hierarchical solve and the full solve where the finest mesh is used for all cell problems, at each $x_1$ and the relative errors between them,
where relative errors are calculated by $\frac{100 |\kappa^*_{full}-\kappa^*_{hier}|}{\kappa^*_{full}}$ with obvious notations for $a=1$ and $a=0.1$ respectively.
The results show clearly that the effective coefficients obtained from hierarchical algorithm are very closed to the reference effective coefficients. We can see from the tables that relatively large errors occur at the highest level macroscopic points where more than one layer of corrections is performed, i.e. the corrector itself is corrected by the solution at a macroscopic point belonging to a lower level. 
We note that the error for the case $a=0.1$ is much smaller as
 the change of $\kappa_i$ in $x$ is much smaller. That is, large Lipschitz constants in Assumption \ref{Lipschitz} tend to result in large errors. 
The results in Tables \ref{table:a1} and \ref{table:a01} are obtained when only one corrector point is employed. If we use more corrector points, the error can be reduced significantly. In
 Table \ref{table:interpolation} 
we show the relative errors, in comparison to the coefficients obtained from the full solve where the finest mesh is used for all the cell problems, for the effective coefficients obtained from the hierarchical solve 
for the two cases where one point and two point interpolations are used. The table shows that the result can be improved by employing two point interpolation. 
\clearpage
\begin{table}
\begin{tabular}{|c|c|c|c|c|c|c|}
  \hline
 \multirow{2}{*} {$x_1$} &  \multicolumn{3} {|c|} {$\kappa^*_{111}(x_1)$}  &  \multicolumn{3} {|c|} {$\kappa^*_{211}(x_1)$} \\
 \cline{2-7}
  & Full solve & Hierarchical solve & Relative errors(\%) & Full solve & Hierarchical solve & Relative errors(\%)\\
  \hline
$0$ & 2.8210 & 2.8210  & 0.0000 &2.8341 & 2.8341 &0.0000\\
[.2em]
$1 \over 16$ & 2.8331 & 2.8266 &0.2305&2.8448 & 2.8432 &0.0560\\
[.2em]
$1 \over 8$&  2.8447 & 2.8406 &0.1425&2.8550&2.8524 &0.0919\\
[.2em]
$3 \over 16$& 2.8557  & 2.8590&0.1160 &2.8649&2.8653&0.0163\\
[.2em]
$1 \over 4$&  2.8663 & 2.8639&0.0823&2.8743& 2.8735&0.0304\\
[.2em]
$5 \over 16$ & 2.8763 &2.8688 &0.2621&2.8834 &2.8813&0.0726\\
[.2em]
$3 \over 8$ &  2.8859   & 2.8886& 0.0936&2.8921  & 2.8936&0.0519 \\
[.2em]
$7 \over 16$ & 2.8950  & 2.8996 &0.1606&2.9005&2.9016&0.0376\\
[.2em]
$1 \over 2$ & 2.9036&  2.9036  &0.0000&2.9085&2.9085&0.0000\\
[.2em]
$9 \over 16$ & 2.9119& 2.9076  &0.1449&2.9161&2.9151&0.0336\\
[.2em]
$5 \over 8$ & 2.9197 & 2.9176 &0.0712&2.9234&2.9218&0.0531\\
[.2em]
$11 \over 16$ & 2.9271& 2.9317 &0.1580&2.9303&2.9317 &0.0470\\
[.2em]
$3 \over 4$ &2.9341& 2.9350 &0.0295&2.9369&2.9373&0.0167\\
[.2em]
$13 \over 16$ & 2.9407 & 2.9382 &0.0851&2.9431 & 2.9428&0.0096\\
[.2em]
$7 \over 8$ & 2.9470& 2.9484&0.0479&2.9490&2.9501&0.0387\\
[.2em]
$15 \over 16$ & 2.9528& 2.9557 &0.0980&2.9546&2.9552&0.0220\\
[.2em]
$1$ & 2.9583 & 2.9583 &0.0000&2.9598&2.9598&0.0000\\
  \hline
\end{tabular} \\
\caption{ a = 1, the effective coefficients $\kappa_{111}^*(x_1)$ and $\kappa_{211}^*(x_1)$ computed by full mesh reference and hierarchical solve along with percentage relative errors between those.}
\label{table:a1}
\end{table}

\begin{table}
\begin{tabular}{|c|c|c|c|c|c|c|}
  \hline
 \multirow{2}{*} {$x_1$} &  \multicolumn{3} {|c|} {$\kappa^*_{111}(x_1)$}  &  \multicolumn{3} {|c|} {$\kappa^*_{211}(x_1)$} \\
 \cline{2-7}
  & Full solve & Hierarchical solve & Relative errors(\%) & Full solve & Hierarchical solve & Relative errors(\%)\\
  \hline
$0$ & 2.8210 & 2.8210  & 0.0000 &2.8341 & 2.8341 &0.0000\\
[.2em]
$1 \over 16$ & 2.8222 & 2.8215 &0.0241&2.8352 & 2.8350 &0.0059\\
[.2em]
$1 \over 8$&  2.8235 & 2.8230 &0.0161&2.8363&2.8360 &0.0102\\
[.2em]
$3 \over 16$& 2.8247  & 2.8250&0.0125 &2.8373&2.8374&0.0020\\
[.2em]
$1 \over 4$&  2.8259 & 2.8256&0.0112&2.8384& 2.8383&0.0037\\
[.2em]
$5 \over 16$ & 2.8271 &2.8261 &0.0347&2.8395 &2.8392&0.0095\\
[.2em]
$3 \over 8$ &  2.8283   & 2.8288& 0.0154&2.8405  & 2.8408&0.0081 \\
[.2em]
$7 \over 16$ & 2.8295  & 2.8302 &0.0232&2.8416&2.8418&0.0056\\
[.2em]
$1 \over 2$ & 2.8307&  2.8307  & 0.0000 &2.8427&2.8427&0.0000\\
[.2em]
$9 \over 16$ & 2.8319& 2.8313  &0.0230&2.8437&2.8435&0.0056\\
[.2em]
$5 \over 8$ & 2.8331 & 2.8327 &0.0150&2.8448&2.8445&0.0096\\
[.2em]
$11 \over 16$ & 2.8343& 2.8352 &0.0327&2.8458&2.8461 &0.0090\\
[.2em]
$3 \over 4$ &2.8355& 2.8357 &0.0100&2.8468&2.8469&0.0035\\
[.2em]
$13 \over 16$ & 2.8366 & 2.8363 &0.0124&2.8479 & 2.8478&0.0019\\
[.2em]
$7 \over 8$ & 2.8378& 2.8382&0.0144&2.9490&2.8492&0.0093\\
[.2em]
$15 \over 16$ & 2.8390& 2.8396 &0.0221&2.8499&2.8501&0.0054\\
[.2em]
$1$ & 2.8401 & 2.8401 &0.0000&2.8510&2.8510&0.0000\\ \hline
\end{tabular} \\
\caption{ a = .1, the effective coefficients $\kappa_{111}^*(x_1)$ and $\kappa_{211}^*(x_1)$ computed by full mesh reference and hierarchical solve along with percentage relative errors between those.}
\label{table:a01}
\end{table}

%

\begin{table}{1-pt interpolation \ \ \; \; \; \; \; \; \; \; \; \; \; \ \ 2-pt interpolation} 
\centering
\begin{tabular}{ll}
\\
\begin{tabular}{|c|c|c|}
 \hline
   \multirow{2}{*} {$x_1$} & \multicolumn{2}{|c|} {Relative Errors (\%)} \\
 \cline{2-3}
 & $\kappa^*_{111}$ & $\kappa^*_{211}$ \\
  \hline
$1 \over 16$ & 0.2305&0.0560\\
[.2em]
$1 \over 8$&0.1425&0.0919\\
[.2em]
$3 \over 16$& 0.1160 &0.0163\\
[.2em]
$1 \over 4$&0.0823&0.0304\\
[.2em]
$5 \over 16$ &0.2621&0.0726\\
[.2em]
$3 \over 8$ & 0.0936&0.0519 \\
[.2em]
$7 \over 16$  &0.1606&0.0376\\
[.2em]
$9 \over 16$  &0.1449&0.0336\\
[.2em]
$5 \over 8$  &0.0712&0.0531\\
[.2em]
$11 \over 16$  &0.1580&0.0470\\
[.2em]
$3 \over 4$ &0.0295&0.0167\\
[.2em]
$13 \over 16$ &0.0851&0.0096\\
[.2em]
$7 \over 8$ &0.0479&0.0387\\
[.2em]
$15 \over 16$ &0.0980&0.0220\\
\hline
\end{tabular} 
&\ \ \ \ \ \ \ \ \ \
\begin{tabular}{|c|c|c|}
   \hline
   \multirow{2}{*} {$x_1$} & \multicolumn{2}{|c|} {Relative Errors (\%)} \\
 \cline{2-3}
 & $\kappa^*_{111}$ & $\kappa^*_{211}$ \\
  \hline
$1 \over 16$ & 0.0072&0.0022\\
[.2em]
$1 \over 8$&0.0093&0.0030\\
[.2em]
$3 \over 16$& 0.0100 &0.0026\\
[.2em]
$1 \over 4$&0.0070&0.0013\\
[.2em]
$5 \over 16$ &0.0081&0.0021\\
[.2em]
$3 \over 8$ & 0.0063&0.0020 \\
[.2em]
$7 \over 16$  &0.0042&0.0013\\
[.2em]
$9 \over 16$  &0.0026&0.0008\\
[.2em]
$5 \over 8$  &0.0032&0.0011\\
[.2em]
$11 \over 16$  &0.0034&0.0009\\
[.2em]
$3 \over 4$ &0.0022&0.0004\\
[.2em]
$13 \over 16$ &0.0027&0.0007\\
[.2em]
$7 \over 8$ &0.0020&0.0007\\
[.2em]
$15 \over 16$ &0.0014 &0.0004\\
 \hline
\end{tabular} 
\end{tabular}
\caption{Percentage relative errors between full mesh reference solve and hierarchical solve when a = 1.}
\label{table:interpolation}
\end{table}

\section{Proof of homogenization convergence}

In this section, we prove rigorously  the homogenization convergence, i.e. the convergence of the solution of the two scale equation \eqref{eq:main3'} to the solution of the homogenized equation (\ref{eq:main15}).
Throughout this section, we denote the spaces $L^2(\Omega)$ and $H_0^1(\Omega)$ as $H$ and $V$ respectively.
We recall the two-scale multi-continuum system
\begin{equation}
\label{eq:main67}
 \begin{split}
{\mathcal C}_{11}^\epsilon{\partial u_1^\epsilon (t,x) \over \partial t}-\text{div}(\kappa_1^\epsilon(x)\nabla u_1^\epsilon (t,x)) - {1 \over \epsilon^2}Q^\epsilon(x)(u_2^\epsilon(t,x)
-u_1^\epsilon(t,x)) = q,
 \end{split}
 \end{equation}
\begin{equation}
\label{eq:main68}
 \begin{split}
{\mathcal C}_{22}^\epsilon{\partial u_2^\epsilon (t,x) \over \partial t}-\text{div}(\kappa_2^\epsilon(x) \nabla u_2^\epsilon (t,x)) - {1 \over \epsilon^2}Q^\epsilon(x)(u_1^\epsilon (t,x)-u_2^\epsilon(t,x)) = q,
 \end{split}
 \end{equation}
 We have the following theorem.
\begin{lemma}
\label{unibound}
The solution $(u_1^\epsilon$, $u_2^\epsilon)$ of (\ref{eq:main67}) and (\ref{eq:main68}) are uniformly bounded in $L^\infty(0,T;H)$ and $L^2(0,T;V)$.
\end{lemma}
\begin{proof}
  Multiplying $\phi_1$ and $\phi_2$ $\in V$ to (\ref{eq:main67}) and (\ref{eq:main68}) respectively and integrating over $\Omega$, one has
  \begin{equation}
\label{eq:main68'}
 \begin{split}
\int_{\Omega}{\mathcal C}_{11}^\epsilon{\partial u_1^\epsilon  \over \partial t}\phi_1 \mathrm{d}x+ \int_{\Omega} \kappa_1^\epsilon\nabla u_1^\epsilon \cdot \nabla \phi_1 \mathrm{d}x
-\int_{\Omega} {1 \over \epsilon^2}Q^\epsilon(u_2^\epsilon-u_1^\epsilon)\phi_1\mathrm{d}x = \int_{\Omega}q \phi_1 \mathrm{d}x,\\
\int_{\Omega}{\mathcal C}_{22}^\epsilon{\partial u_2^\epsilon  \over \partial t} \phi_2 \mathrm{d}x+ \int_{\Omega} \kappa_2^\epsilon\nabla u_2^\epsilon \cdot \nabla \phi_2 \mathrm{d}x
-\int_{\Omega} {1 \over \epsilon^2}Q^\epsilon(u_1^\epsilon-u_2^\epsilon)\phi_2 \mathrm{d}x = \int_{\Omega}q \phi_2 \mathrm{d}x.
 \end{split}
 \end{equation}
Summing these equations, we get
\begin{equation}
\label{eq:main68''}
 \begin{split}
\int_{\Omega}{\mathcal C}_{11}^\epsilon{\partial u_1^\epsilon(t)  \over \partial t}\phi_1 \mathrm{d}x + \int_{\Omega} \kappa_1^\epsilon\nabla u_1^\epsilon(t) \cdot \nabla \phi_1 \mathrm{d}x
-\int_{\Omega} {1 \over \epsilon^2}Q^\epsilon(u_2^\epsilon(t)-u_1^\epsilon(t))\phi_1 \mathrm{d}x\\
+\int_{\Omega}{\mathcal C}_{22}^\epsilon{\partial u_2^\epsilon  \over \partial t}(t) \phi_2 \mathrm{d}x+ \int_{\Omega} \kappa_2^\epsilon\nabla u_2^\epsilon(t) \cdot \nabla \phi_2 \mathrm{d}x
-\int_{\Omega} {1 \over \epsilon^2}Q^\epsilon(u_1^\epsilon(t)-u_2^\epsilon(t))\phi_2 \mathrm{d}x\\
= \int_{Q}q(t) \phi_1\mathrm{d}x+\int_{\Omega}q(t) \phi_2 \mathrm{d}x
 \end{split}
 \end{equation}
 $\forall\phi_1,\phi_2\in V$. 
 Substituting $u_1^\epsilon$ and $u_2^\epsilon$ into  $\phi_1$ and $\phi_2$ in (\ref{eq:main68''}) respectively, we have
 \begin{equation*}
\label{eq:main71}
 \begin{split}
\int_{\Omega}{\mathcal C}_{11}^\epsilon{\partial u_1^\epsilon(t)  \over \partial t} u_1^\epsilon(t) \mathrm{d}x
+ \int_{\Omega}{\mathcal C}_{22}^\epsilon{\partial u_2^\epsilon(t)  \over \partial t} u_2^\epsilon(t) \mathrm{d}x
 + \int_{\Omega} \kappa_1^\epsilon\nabla u_1^\epsilon(t) \cdot \nabla u_1^\epsilon(t) \mathrm{d}x&\\
 + \int_{\Omega} \kappa_2^\epsilon\nabla u_2^\epsilon(t) \cdot \nabla u_2^\epsilon (t) \mathrm{d}x
+{1 \over \epsilon^2} \int_{\Omega} Q^\epsilon(u_2^\epsilon(t)-u_1^\epsilon(t))^2 \mathrm{d}x
= \int_{\Omega}q u_1^\epsilon(t) + \int_{\Omega}q u_2^\epsilon(t) \mathrm{d}x &
 \end{split}
 \end{equation*}
Integrating this equation over $(0,\tau)$, we get
 \begin{equation}
\label{eq:main73}
 \begin{split}
\frac{1}{2}\int_{\Omega} {\mathcal C}_{11}^\epsilon |u_1^\epsilon(\tau,x)|^2 \mathrm{d}x+ \frac{1}{2} \int_{\Omega} {\mathcal C}_{22}^\epsilon |u_2^\epsilon(\tau,x)|^2 \mathrm{d}x
 + \int^\tau_0\int_{\Omega} \kappa_1^\epsilon\nabla u_1^\epsilon \cdot \nabla u_1^\epsilon \mathrm{d}x \mathrm{d}t\\
 + \int^\tau_0\int_{\Omega} \kappa_2^\epsilon\nabla u_2^\epsilon \cdot \nabla u_2^\epsilon \mathrm{d}x \mathrm{d}t
+{1 \over \epsilon^2} \int^\tau_0\int_{\Omega} Q^\epsilon(u_2^\epsilon-u_1^\epsilon)^2 \mathrm{d}x \mathrm{d}t\\
= \int^\tau_0\int_{\Omega}q u_1^\epsilon \mathrm{d}x\mathrm{d}t
+\int^\tau_0 \int_{\Omega}q u_2^\epsilon \mathrm{d}x\mathrm{d}t
+ \frac{1}{2} \int_{\Omega} {\mathcal C}_{11}^\epsilon|u_1^\epsilon(0,x)|^2 \mathrm{d}x
+ \frac{1}{2} \int_{\Omega} {\mathcal C}_{22}^\epsilon|u_2^\epsilon(0,x)|^2 \mathrm{d}x.&
 \end{split}
 \end{equation}
 Therefore, 
  \begin{equation*}
\label{eq:main73'}
 \begin{split}
\frac{1}{2}\int_{\Omega} {\mathcal C}_{11}^\epsilon |u_1^\epsilon(\tau,x)|^2 \mathrm{d}x
+ \frac{1}{2} \int_{\Omega} {\mathcal C}_{22}^\epsilon |u_2^\epsilon(\tau,x)|^2 \mathrm{d}x
 + \int^\tau_0\int_{\Omega} \kappa_1^\epsilon\nabla u_1^\epsilon \cdot \nabla u_1^\epsilon\mathrm{d}x\mathrm{d}t
 + \int^\tau_0\int_{\Omega} \kappa_2^\epsilon\nabla u_2^\epsilon \cdot \nabla u_2^\epsilon \mathrm{d}x\mathrm{d}t   &
\\ \leq c \int^\tau_0\int_{\Omega}|q|^2 \mathrm{d}x\mathrm{d}t
+ \delta \int^\tau_0\int_{\Omega} |u_1^\epsilon|^2  \mathrm{d}x \mathrm{d}t
+ c \int^\tau_0 \int_{\Omega}|q|^2 \mathrm{d}x\mathrm{d}t \\
 +\delta \int^\tau_0\int_{\Omega} |u_2^\epsilon|^2\mathrm{d}x\mathrm{d}t
 + \int_{\Omega} |{\mathcal C}_{11}^\epsilon| |u_1^\epsilon(0,x)|^2 \mathrm{d}x
+ \int_{\Omega} |{\mathcal C}_{22}^\epsilon| |u_2^\epsilon(0,x)|^2 \mathrm{d}x.&
 \end{split}
 \end{equation*}
Using the uniform boundedness from below of $C_{11}^\ep$ and $C_{22}^\ep$, we have
\[
\begin{split}
c\|u_1^\epsilon(\tau,\cdot)\|_H^2+c\|u_2^\epsilon(\tau,\cdot)\|_H^2+\int^\tau_0\int_{\Omega} \kappa_1^\epsilon\nabla u_1^\epsilon \cdot \nabla u_1^\epsilon\mathrm{d}x\mathrm{d}t
 + \int^\tau_0\int_{\Omega} \kappa_2^\epsilon\nabla u_2^\epsilon \cdot \nabla u_2^\epsilon \mathrm{d}x\mathrm{d}t  &\\
 \le c+\delta\int_0^T\|u_1^\epsilon(t,\cdot)\|_H^2dt+\delta\int_0^T\|u_2^\epsilon(t,\cdot)\|_H^2dt.
 \end{split}
 \]
 Choosing $\delta$ sufficiently small, we deduce that $u_1^\epsilon$ and $u_2^\epsilon$ are uniformly bounded in $L^\infty(0,T;H)$ and $L^2(0,T;V)$. 
 \end{proof}
  Note that because of the 5th term of equation (\ref{eq:main73}), $\displaystyle\lim_{\epsilon \to 0}u_1^\epsilon = \lim_{\epsilon \to 0}u_2^\epsilon$.
 Thus, there exist subsequences of $u_1^\epsilon$ and $u_2^\epsilon$, which we still denote by $u_1^\epsilon$ and $u_2^\epsilon$ , and $u_0$ such that
\begin{equation*}
\label{eq:main74}
 \begin{split}
u_1^{\epsilon}, u_2^{\epsilon} \wc u_0 \enspace \textrm{in} \enspace L^2(0,T;V). 
 \end{split}
 \end{equation*}
 Recall that $({N^i_1},{N^i_2})\in W$ is the solution of cell problem.
 \begin{equation}
  \label{eq:dualcell}
 \begin{split}
   &\text{div}_y(\kappa_1(x,y) (e^i+\nabla_y {N^i_1}(x,y))) +
   Q(x,y) ({N^i_2}(x,y)- {N^i_1}(x,y)) =0 \qquad  \\
   &\text{div}_y(\kappa_2(x,y) (e^i+\nabla_y {N^i_2}(x,y))) +
   Q(x,y) ( {N^i_1}(x,y)- {N^i_2}(x,y)) =0.\\
 \end{split}
\end{equation}
We assume that $N_1^i$ and $N_2^i$ are sufficiently smooth with respect to both $x$ and $y$.
Let $\omega_1(x) = \frac{x_i}{\epsilon} + N^i_1(x,\frac{x}{\epsilon})$ and $\omega_2(x) = \frac{x_i}{\epsilon} + N^i_2(x,\frac{x}{\epsilon})$. 
We define $\omega_1^\epsilon$ and $\omega_2^\epsilon$ as 
 \begin{equation*}
  \label{eq:main81}
 \begin{split}
\omega_1^\epsilon(x) = \epsilon \omega_1(x,\frac{x}{\epsilon}),\ \ \omega_2^\epsilon(x) = \epsilon \omega_2(x,\frac{x}{\epsilon}) .
 \end{split}
\end{equation*}
Assuming that $\kappa_1$, $\kappa_2$, $N_1^i$ and $N_2^i$ are sufficiently smooth, for all $\psi_1, \psi_2\in V$ we have 
 \begin{equation}
\label{eq:main81'''}
 \begin{split}
-\int_\Omega&\text{div}(\kappa_1^\epsilon(x)\nabla \omega^\epsilon_1(x)) \psi_1(x) \mathrm{d}x
- {1 \over \epsilon^2}\int_\Omega Q^\epsilon(x)( \omega_2^\epsilon(x)- \omega_1^\epsilon(x))\psi_1(x)\mathrm{d}x\\
= -&\frac{1}{\epsilon}\int_\Omega 
\text{div}_y(\kappa_1(x,\frac{x}{\epsilon}) (e^i + \nabla_y N_1^i (x,\frac{x}{\epsilon})))\psi_1(x) \mathrm{d}x
- {1 \over \epsilon}\int_\Omega Q(x,\frac{x}{\epsilon})(N^i_2(x,\frac{x}{\epsilon})-N^i_1(x,\frac{x}{\epsilon}))\psi_1(x)\mathrm{d}x\\
-&\epsilon\int_\Omega\div_x (\kappa_1(x,\frac{x}{\epsilon})\nabla_x(N_1^i(x,\frac{x}{\epsilon})))\psi_1(x)\mathrm{d}x
-\int_\Omega \div_x(\kappa_1(x,\frac{x}{\epsilon})(e^i + \nabla_y N_1^i (x,\frac{x}{\epsilon})))\psi_1(x)\mathrm{d}x
\\-&\int_\Omega\div_y(\kappa_1(x,\frac{x}{\epsilon})\nabla_x N_1^i(x,\frac{x}{\epsilon}))\psi_1(x)\mathrm{d}x\\
=-&\epsilon\int_\Omega\div_x (\kappa_1(x,\frac{x}{\epsilon})\nabla_xN_1^i(x,\frac{x}{\epsilon}))\psi_1(x)\mathrm{d}x 
- \int_\Omega \div_x(\kappa_1(x,\frac{x}{\epsilon})(e^i + \nabla_y N_1^i (x,\frac{x}{\epsilon})))\psi_1(x)\mathrm{d}x
\\-& \int_\Omega\div_y(\kappa_1(x,\frac{x}{\epsilon})\nabla_x N_1^i(x,\frac{x}{\epsilon}))\psi_1(x)\mathrm{d}x
 \end{split}
 \end{equation}
 and
 \begin{equation}
\label{eq:main81'''*}
 \begin{split}
-\int_\Omega&\text{div}(\kappa_2^\epsilon(x)\nabla \omega^\epsilon_2(x)) \psi_2(x)\mathrm{d}x
- {1 \over \epsilon^2}\int_\Omega Q^\epsilon(x)( \omega_1^\epsilon(x)- \omega_2^\epsilon(x))\psi_2(x)\mathrm{d}x\\
= -&\frac{1}{\epsilon}\int_\Omega 
\text{div}_y(\kappa_2(x,\frac{x}{\epsilon})(e^i + \nabla_y N_2^i (x,\frac{x}{\epsilon}))) \psi_2(x)\mathrm{d}x
- {1 \over \epsilon}\int_\Omega Q(x,y)(N^i_1(x,\frac{x}{\epsilon})-N^i_2(x,\frac{x}{\epsilon})) \psi_2(x)\mathrm{d}x\\
-&\epsilon\int_\Omega\div_x (\kappa_2(x,\frac{x}{\epsilon})\nabla_x(N_2^i(x,\frac{x}{\epsilon})))\psi_2(x) \mathrm{d}x
-\int_\Omega \div_x(\kappa_2(x,\frac{x}{\epsilon})(e^i + \nabla_y N_2^i (x,\frac{x}{\epsilon})))\psi_2(x)\mathrm{d}x
\\-&\int_\Omega\div_y(\kappa_2(x,\frac{x}{\epsilon})\nabla_x N_2^i(x,\frac{x}{\epsilon}))\psi_2(x) \mathrm{d}x\\
=-&\epsilon\int_\Omega\div_x (\kappa_2(x,\frac{x}{\epsilon})\nabla_xN_2^i(x,\frac{x}{\epsilon}))\psi_2(x)\mathrm{d}x
- \int_\Omega \div_x(\kappa_2(x,\frac{x}{\epsilon}) (e^i + \nabla_y N_2^i (x,\frac{x}{\epsilon})))\psi_2(x)\mathrm{d}x
\\-& \int_\Omega\div_y(\kappa_2(x,\frac{x}{\epsilon})\nabla_x N_2^i(x,\frac{x}{\epsilon}))\psi_2(x)\mathrm{d}x
 \end{split}
 \end{equation}
due to (\ref{eq:dualcell}). 
Let $\phi_1(x) = \phi(x)\omega_1^\epsilon(x)$,  
 $\phi_2(x) = \phi(x)\omega_2^\epsilon(x)$ where $\phi \in \mathcal{C}^\infty_0(\Omega)$ in (\ref{eq:main68'}), we have
\begin{equation}
\label{eq:main82*}
 \begin{split}
\int_{\Omega}{\mathcal C}_{11}^\epsilon{\partial u_1^\epsilon  \over \partial t}\phi\omega_1^\epsilon \mathrm{d}x
+\int_{\Omega}{\mathcal C}_{22}^\epsilon{\partial u_2^\epsilon  \over \partial t} \phi\omega_2^\epsilon \mathrm{d}x
+ \int_{\Omega} \kappa_1^\epsilon\nabla u_1^\epsilon \cdot \nabla (\phi\omega_1^\epsilon) \mathrm{d}x
+\int_{\Omega} \kappa_2^\epsilon\nabla u_2^\epsilon \cdot \nabla (\phi\omega_2^\epsilon) \mathrm{d}x\\
+\int_{\Omega} {1 \over \epsilon^2}Q^\epsilon(u_1^\epsilon-u_2^\epsilon)(\omega_1^\epsilon-\omega_2^\epsilon)\phi \mathrm{d}x
= \int_{\Omega}q \phi\omega_1^\epsilon  \mathrm{d}x+\int_{\Omega}q \phi\omega_2^\epsilon \mathrm{d}x.
 \end{split}
 \end{equation}
Let $\psi_1(x)$ and $\psi_2(x)$ in (\ref{eq:main81'''}) and (\ref{eq:main81'''*}) be $\phi u_1^\epsilon$ and $\phi u_2^\epsilon$ respectively. We have
 \begin{equation}
\label{eq:main82**}
 \begin{split}
\int_{\Omega} \kappa_1^\epsilon\nabla \omega_1^\epsilon \cdot \nabla (\phi u_1^\epsilon) \mathrm{d}x
+ \int_{\Omega} \kappa_2^\epsilon\nabla \omega_2^\epsilon \cdot \nabla(\phi u_2^\epsilon) \mathrm{d}x
+\int_{\Omega} {1 \over \epsilon^2}Q^\epsilon(\omega_1^\epsilon-\omega_2^\epsilon)(u_1^\epsilon-u_2^\epsilon)\phi \mathrm{d}x\\
= -\epsilon \int_{\Omega} \div_x (\kappa_1(x,\xoe) \nabla_x N_1^i(x,\xoe)) \phi u_1^\epsilon\mathrm{d}x
- \int_{\Omega} \div_x(\kappa_1(x,\xoe)(e^i + \nabla_y N_1^i (x,\xoe)))\phi u_1^\epsilon \mathrm{d}x\\
-\int_{\Omega} \div_y( \kappa_1(x,\xoe)\nabla_x N_1^i(x,\xoe) )\phi u_1^\epsilon \mathrm{d}x
- \epsilon \int_{\Omega} \div_x( \kappa_2(x,\xoe)\nabla_x N_2^i(x,\xoe))\phi u_2^\epsilon \mathrm{d}x\\
- \int_{\Omega} \div_x(\kappa_2(x,\xoe)(e^i + \nabla_y N_2^i (x,\xoe))\phi u_2^\epsilon \mathrm{d}x
- \int_{\Omega} \div_y (\kappa_2(x,\xoe)\nabla_x N_2^i(x,\xoe)) \phi u_2^\epsilon \mathrm{d}x
 \end{split}
 \end{equation}
 Let $\psi\in C^\infty_0(0,T)$. We multiply \eqref{eq:main82*} and \eqref{eq:main82**} by $\psi$ and intergrate over $(0,T)$ with respect to $t$. After subtracting the resulting equations by each other, we obtain
 \begin{equation}
\label{eq:main82'}
 \begin{split}
\int_0^T\int_\Omega {\mathcal C}_{11}^\epsilon{\partial u_1^\epsilon\over \partial t} \phi\psi\omega_1^\epsilon \mathrm{d}x\dt
+ \int_0^T\int_\Omega \kappa_1^\epsilon \nabla u_1^\epsilon\cdot \nabla \phi \omega_1^\epsilon \psi \mathrm{d}x\dt
- \int_0^T\int_\Omega \kappa_1^\epsilon \nabla \omega_1^\epsilon\cdot \nabla\phi u_1^\epsilon\psi \mathrm{d}x \dt\\
+ \int_0^T\int_\Omega {\mathcal C}_{22}^\epsilon{\partial u_2^\epsilon \over \partial t} \phi\psi\omega_2^\epsilon \mathrm{d}x\dt
+\int_0^T \int_\Omega \kappa_2^\epsilon \nabla u_2^\epsilon\cdot \nabla\phi \omega_2^\epsilon\psi \mathrm{d}x\dt
- \int_0^T\int_\Omega \kappa_2^\epsilon \nabla \omega_2^\epsilon\cdot\nabla\phi u_2^\epsilon\psi\mathrm{d}x\dt \\
= \int_0^T\int_\Omega q \phi \omega_1^\epsilon\psi \dx\dt+ \int_0^T\int_\Omega q \phi \omega_2^\epsilon\psi \mathrm{d}x\dt \\
+\epsilon\int_0^T \int_{\Omega} \div_x (\kappa_1(\cdot,{\cdot\over\epsilon}) \nabla_x N_1^i(\cdot,{\cdot\over\epsilon}) \phi u_1^\epsilon\psi\mathrm{d}x\dt
+ \int_0^T\int_{\Omega} \div_x(\kappa_1(\cdot,{\cdot\over\epsilon})(e^i + \nabla_y N_1^i(\cdot,{\cdot\epsilon}) ))\phi u_1^\epsilon\psi \mathrm{d}x\dt\\
+\int_0^T\int_{\Omega} \div_y( \kappa_1(\cdot,{\cdot\over\epsilon})\nabla_x N_1^i(\cdot,{\cdot\over\epsilon} )\phi u_1^\epsilon \psi\mathrm{d}x\dt
+ \epsilon \int_0^T\int_{\Omega} \div_x( \kappa_2(\cdot,{\cdot\over\epsilon})\nabla_x N_2^i(\cdot,{\cdot\over\epsilon}))\phi u_2^\epsilon\psi \mathrm{d}x\dt\\
+ \int_0^T\int_{\Omega} \div_x(\kappa_2(\cdot,{\cdot\over\epsilon})(e^i + \nabla_y N_2^i (\cdot,{\cdot\over\epsilon})))\phi u_2^\epsilon\psi \mathrm{d}x\dt
+\int_0^T \int_{\Omega} \div_y (\kappa_2(\cdot,{\cdot\over\epsilon})\nabla_x N_2^i(\cdot,{\cdot\over\epsilon}) \phi u_2^\epsilon\psi \mathrm{d}x\dt.
\end{split}
\end{equation}
We have the following lemma.
\begin{lemma}
\label{lemma2}
The functions $\int_0^T\psi(t)u_1^\ep(x,t)dt$ and  $\int_0^T\psi(t)u_2^\ep(x,t)dt$ converge strongly in $H$ to $\int_0^T\psi(t) u_0(x,t)dt$.
\end{lemma}
{\it Proof\ \ } This is the standard result in Jikov et al. \cite{jikov2012homogenization}. As $u_1^\ep$ is uniformly bounded in $L^2(0,T;V)$, $\int_0^T\psi(t)u_1^\ep(x,t)\dt$ is uniformly bounded in $V$ when $\epsilon\to 0$. Thus we can extract a subsequence which converges weakly in $V$ and strongly in $H$. As for all $\phi\in C^\infty_0(\Omega)$, 
\[
\int_\Omega\int_0^T\psi(t)u_1^\ep(x,t)\phi(x)\dt\dx\to \int_\Omega\int_0^T\psi(t)u_0(x,t)\phi(x)\dt\dx,
\]
the limit is $\int_0^T\psi(t)u_0(x,t)\dt$.\hfill$\Box$\\
We have
\[
\begin{split}
\int_0^T\int_\Omega C_{11}^\ep{\partial u_1^\ep\over\partial t}\phi\psi\omega_1^\ep \dx\dt=-\int_\Omega C_{11}^\ep\left(\int_0^T u_1^\ep{\partial\psi\over\partial t}dt\right)\phi\omega_1^\ep \dx.
\end{split}
\]
As $C_{11}^\ep$ converges weakly to $\int_YC_{11}(x,y)\dy$ in $H$, $\int_0^Tu_1^\ep{\partial\psi\over\partial t}\dt$ converges weakly to $\int_0^Tu_0{\partial\psi\over\partial t}\dt$ in $V$, we have 
\[
\begin{split}
\lim_{\ep\to 0}\int_0^T\int_\Omega C_{11}^\ep{\partial u_1^\ep\over\partial t}\phi\psi\omega_1^\ep \dx\dt=
-\int_0^T\int_\Omega\left(\int_YC_{11}(x,y)dy\right)u_0{\partial\psi\over\partial t}\phi x_i \dx\dt \\
=\int_0^T\int_\Omega\left(\int_YC_{11}(x,y)dy\right){\partial u_0\over\partial t}\psi\phi x_i \dx\dt.
\end{split}
\]
We note that
\begin{equation*}
\label{eq:main82''}
 \begin{split}
 \kappa_1^\epsilon(x) \nabla \omega_1^\epsilon(x) 
 = \kappa_1(x,\frac{x}{\epsilon})\big((e^i + \nabla_y N_1^i (x,\frac{x}{\epsilon}))+ \epsilon \nabla_x N^i_1(x,\frac{x}{\epsilon})\big),\\
 \kappa_2^\epsilon(x) \nabla \omega_2^\epsilon(x)
 = \kappa_2(x,\frac{x}{\epsilon})\big((e^i + \nabla_yN_2^i (x,\frac{x}{\epsilon}))+ \epsilon \nabla_x N^i_2(x,\frac{x}{\epsilon})\big).
  \end{split}
 \end{equation*}
 Also, note that due to $Y$-periodicity of $\kappa$ and $N^i$, we have
\begin{equation*}
\label{eq:main82'''}
 \begin{split}
&\kappa_1(x,\frac{x}{\epsilon}) (e^i + \nabla_y N_1^i (x,\frac{x}{\epsilon}))
 \wc \int_Y \kappa_1(x, y)(e^i + \nabla_y N_1^i (x,y)) \mathrm{d}y\\
&\kappa_2(x,\frac{x}{\epsilon}) (e^i + \nabla_y N_2^i (x,\frac{x}{\epsilon}))
 \wc \int_Y \kappa_2(x, y)(e^i + \nabla_y N_2^i (x,y)) \mathrm{d}y
\enspace \textrm{in} \enspace H 
  \end{split}
 \end{equation*}
We observe that $x_i + \epsilon N^i_1 \rightarrow x_i$ strongly in $H$.
Passing to the limit in the left hand side of (\ref{eq:main82'}), 
we obtain from Lemma \ref{lemma2},
\begin{equation}
\label{eq:main82''''}
 \begin{split}
 \int_0^T \int_\Omega \left(\int_Y {\mathcal C}_{11} \mathrm{d}y\right){\partial u_0 \over \partial t} \phi \psi x_i \mathrm{d}x \mathrm{d}t
+ \displaystyle \lim_{\epsilon \to 0 }
\int_0^T\int_\Omega \kappa_1^\epsilon \nabla u_1^\epsilon \cdot\nabla \phi\psi x_i \mathrm{d}x \mathrm{d}t\\
- \int_0^T\int_\Omega \int_Y \kappa_1(e^i + \nabla_y N_1^i) \mathrm{d}y \cdot\nabla \phi \psi u_0 \mathrm{d}x\mathrm{d}t
+ \int_0^T\int_\Omega \left(\int_Y {\mathcal C}_{22} \mathrm{d}y\right){\partial u_0 \over \partial t} \phi \psi x_i \mathrm{d}x\mathrm{d}t\\
+ \displaystyle \lim_{\epsilon \to 0 } 
\int_0^T\int_\Omega \kappa_2^\epsilon \nabla u_2^\epsilon \cdot \nabla \phi \psi x_i \mathrm{d}x\mathrm{d}t
- \int_0^T\int_\Omega \int_Y \kappa_2 (e^i + \nabla_y N_2^i) \mathrm{d}y \cdot\nabla \phi \psi u_0 \mathrm{d}x\mathrm{d}t.
 \end{split}
 \end{equation}
 Note that  $\int_{\Omega} q \phi \omega_k^\epsilon \mathrm{d}x \rightarrow \int_\Omega  q \phi x_i \mathrm{d}x$
 since $\omega_k^\epsilon\phi \rightarrow x_i\phi$ in $H$ weakly.
 Thus from  (\ref{eq:main82'}) and (\ref{eq:main82''''}), we have
 \begin{equation}
\label{eq:main84}
 \begin{split}
 \int_0^T\int_\Omega\left( \int_Y {\mathcal C}_{11} \mathrm{d}y\right){\partial u_0 \over \partial t} \phi \psi x_i \mathrm{d}x\mathrm{d}t
+  \displaystyle \lim_{\epsilon \to 0 }
\int_0^T\int_\Omega \kappa_1^\epsilon \nabla u_1^\epsilon \cdot\nabla \phi \psi x_i\mathrm{d}x\mathrm{d}t\\
-\int_0^T \int_\Omega\int_Y \kappa_1(e^i + \nabla_y N_1^i) \mathrm{d}y\cdot(\nabla \phi) \psi u_0 \mathrm{d}x\mathrm{d}t
+ \int_0^T\int_\Omega \left(\int_Y {\mathcal C}_{22} \mathrm{d}y\right){\partial u_0 \over \partial t} \phi \psi x_i \mathrm{d}x\mathrm{d}t\\
+  \displaystyle \lim_{\epsilon \to 0 }
\int_0^T\int_\Omega \kappa_2^\epsilon \nabla u_2^\epsilon \cdot\nabla \phi\psi x_i\mathrm{d}x\mathrm{d}t
- \int_0^T\int_\Omega \int_Y \kappa_2 (e^i + \nabla_y N_2^i) \mathrm{d}y \cdot\nabla \phi\psi u_0\mathrm{d}x\mathrm{d}t\\
=2 \int_0^T\int_\Omega  q \phi \psi x_i \mathrm{d}x \mathrm{d}t
 + \displaystyle \lim_{\epsilon \to 0} 
 \int_0^T\int_{\Omega} \div_x(\kappa_1(\cdot,{\cdot\over\epsilon})(e^i + \nabla_y N_1^i(\cdot,{\cdot\over\epsilon}))) \phi \psi u_1^\epsilon \mathrm{d}x\mathrm{d}t\\
 + \displaystyle \lim_{\epsilon \to 0} 
 \int_0^T\int_{\Omega} \div_x(\kappa_2(\cdot,{\cdot\over\epsilon})(e^i + \nabla_y N_2^i(\cdot,{\cdot\over\ep}))) \phi \psi u_2^\epsilon \mathrm{d}x\mathrm{d}t
 \end{split}
 \end{equation}
Let $\phi_1$ and $\phi_2$ in (\ref{eq:main68'}) be $\phi x_i$ where $\phi\in C^\infty_0(\Omega)$. Adding the two equations, we have
\begin{equation*}
 \begin{split}
 \int_0^T\int_{\Omega}{\mathcal C}_{11}^\epsilon{\partial u_1^\epsilon  \over \partial t} \phi \psi x_i \mathrm{d}x\mathrm{d}t
+  \int_0^T\int_{\Omega}{\mathcal C}_{22}^\epsilon{\partial u_2^\epsilon  \over \partial t} \phi \psi x_i \mathrm{d}x\mathrm{d}t
+ \int_0^T \int_{\Omega} \kappa_1^\epsilon\nabla u_1^\epsilon \cdot \nabla (\phi x_i) \psi \mathrm{d}x \mathrm{d}t\\
+  \int_0^T\int_{\Omega} \kappa_2^\epsilon\nabla u_2^\epsilon \cdot \nabla (\phi x_i)\psi \mathrm{d}x\mathrm{d}t
= 2  \int_0^T\int_{\Omega}q \phi\psi  x_i \mathrm{d}x\mathrm{d}t
 \end{split}
 \end{equation*} 
 Passing to the limit, we obtain
\begin{equation}
\label{eq:main84'''}
 \begin{split}
\int_0^T\int_\Omega\left( \int_Y {\mathcal C}_{11} \mathrm{d}y\right){\partial u_0 \over \partial t} \phi\psi x_i \mathrm{d}x\mathrm{d}t
+ \int_0^T\int_\Omega\left( \int_Y {\mathcal C}_{22} \mathrm{d}y\right){\partial u_0 \over \partial t} \phi\psi  x_i \mathrm{d}x\mathrm{d}t\\
+\displaystyle \lim_{\epsilon \to 0} 
\int_0^T\int_\Omega \kappa_1^\epsilon \nabla u_1^\epsilon\nabla\cdot (\phi x_i)\psi  \mathrm{d}x\mathrm{d}t
+ \displaystyle \lim_{\epsilon \to 0} 
\int_0^T\int_\Omega \kappa_2^\epsilon \nabla u_2^\epsilon\cdot \nabla (\phi x_i) \psi \mathrm{d}x\mathrm{d}t
= 2\int_0^T \int_\Omega q (\phi x_i)\psi  \mathrm{d}x\mathrm{d}t
 \end{split}
 \end{equation}
 Using (\ref{eq:main84}) and (\ref{eq:main84'''}), one obtains
 \begin{equation*}
\label{eq:main84''''}
 \begin{split}
&\displaystyle \lim_{\epsilon \to 0} 
\int_0^T \int_{\Omega} \kappa_1^\epsilon \nabla u_1^\epsilon \cdot e^i \phi \psi \mathrm{d}x\mathrm{d}t
+ \displaystyle \lim_{\epsilon \to 0} 
\int_0^T\int_{\Omega} \kappa_2^\epsilon \nabla u_2^\epsilon \cdot e^i \phi \psi  \mathrm{d}x\mathrm{d}t\\
&= - \int_0^T\int_\Omega \int_Y \kappa_1(e^i + \nabla_y N_1^i) \mathrm{d}y \cdot\nabla \phi \psi u_0 \mathrm{d}x\mathrm{d}t
- \int_0^T\int_\Omega \int_Y \kappa_2(e^i + \nabla_y N_2^i) \mathrm{d}y \cdot \nabla \phi \psi u_0 \mathrm{d}x\mathrm{d}t\\
&- \displaystyle \lim_{\epsilon \to 0}
\int_0^T \int_{\Omega} \div_x(\kappa_1(x,\frac{x}{\epsilon})(e^i + \nabla_y N_1^i(x,\frac{x}{\epsilon}))) \phi \psi u_1^\epsilon \mathrm{d}x\mathrm{d}t
- \displaystyle \lim_{\epsilon \to 0} 
\int_0^T\int_{\Omega} \div_x(\kappa_2(x,\frac{x}{\epsilon})(e^i + \nabla_y N_2^i(x,\frac{x}{\epsilon}))) \phi\psi  u_2^\epsilon \mathrm{d}x\mathrm{d}t
 \end{split}
 \end{equation*}
 Since $\kappa_1$, $\kappa_2$, $N_1^i$ and $N_2^i$ are independent of $t$, by Lemma \ref{lemma2}, we have
  \begin{equation*}
\label{eq:main84*}
 \begin{split}
&\displaystyle \lim_{\epsilon \to 0} 
\int^T_0\int_{\Omega} \kappa_1^\epsilon \nabla u_1^\epsilon \cdot e^i \phi \psi \mathrm{d}x \dt
+ \displaystyle \lim_{\epsilon \to 0} 
\int^T_0\int_{\Omega} \kappa_2^\epsilon \nabla u_2^\epsilon \cdot e^i \phi\psi  \mathrm{d}x \mathrm{d}t\\
&= -  \int^T_0\int_\Omega \left(\int_Y \kappa_1(e^i + \nabla_y N_1^i\right)\mathrm{d}y\big)\cdot\nabla \phi \psi u_0 \mathrm{d}x\mathrm{d}t
- \int^T_0\int_\Omega \left(\int_Y \kappa_2(e^i + \nabla_y N_2^i)\mathrm{d}y\right)\cdot\nabla \phi \psi u_0 \mathrm{d}x\mathrm{d}t\\
&- \displaystyle 
\int^T_0\int_{\Omega} \div_x \left(\int_Y \kappa_1(e^i + \nabla_y N_1^i)\mathrm{d}y\right) \phi \psi u_0 \mathrm{d}x\mathrm{d}t
- \displaystyle  
\int^T_0\int_{\Omega} \div_x  \left(\int_Y \kappa_2(e^i + \nabla_y N_2^i)\mathrm{d}y\right) \phi\psi  u_0 \mathrm{d}x\mathrm{d}t.
 \end{split}
 \end{equation*}
 Therefore, we have
\begin{equation*}
\label{eq:main84**}
 \begin{split}
&\displaystyle \lim_{\epsilon \to 0}
\int^T_0 \left(\int_{\Omega} \kappa_1^\epsilon \nabla u_1^\epsilon \cdot e^i \phi  \mathrm{d}x\right)\psi \mathrm{d}t
+ \displaystyle \lim_{\epsilon \to 0} 
\int^T_0 \left(\int_{\Omega} \kappa_2^\epsilon \nabla u_2^\epsilon \cdot e^i \phi \mathrm{d}x\right)\psi\mathrm{d}t\\
&=  \int^T_0\left(\int_\Omega \left(\int_Y \kappa_1(e^i + \nabla_y N_1^i)\mathrm{d}y\right)\cdot\nabla u_0 \phi \mathrm{d}x\right)\psi\mathrm{d}t
+ \int^T_0\left(\int_\Omega \left(\int_Y \kappa_2(e^i + \nabla_y N_2^i)\mathrm{d}y\right)\cdot\nabla u_0 \phi \mathrm{d}x\right)\psi\mathrm{d}t\\
 \end{split}
 \end{equation*}
 From this, we deduce
\begin{equation}
\label{eq:main85}
 \begin{split}
&  \displaystyle \lim_{\epsilon \to 0} \left[ 
  \int_0^T\int_\Omega \kappa_1^\epsilon(x) \nabla u_1^\epsilon(x)\cdot \nabla \phi\psi  \mathrm{d}x \mathrm{d}t
+   \int_0^T\int_\Omega \kappa_2^\epsilon(x) \nabla u_2^\epsilon(x)\cdot \nabla \phi\psi  \mathrm{d}x \mathrm{d}t\right] \\
&=\lim_{\epsilon\to 0}\left[\int_0^T\int_\Omega\kappa_1^\ep\nabla u_1^\epsilon\cdot e^i{\partial\phi\over\partial x_i}\psi dxdt+\int_0^T\int_\Omega \kappa_2^\ep\nabla u_2^\epsilon\cdot e^i{\partial\phi\over\partial x_i}\psi dxdt\right]\\
&=   \int_0^T\int_\Omega \int_Y  \kappa_1(x,y) (\delta_{ij} + {\partial  N^i_1(x,y)\over \partial y_j}) \mathrm{d}y {\partial u_0\over\partial x_j}(x) {\partial\phi\over\partial x_i}\psi  \mathrm{d}x\mathrm{d}t\\
&+  \int_0^T\int_\Omega \int_Y  \kappa_2(x,y) (\delta_{ij} + {\partial  N^i_2(x,y)\over \partial y_j}) \mathrm{d}y {\partial u_0\over\partial x_j}(x) {\partial\phi\over\partial x_i}\psi  \mathrm{d}x \mathrm{d}t 
 \end{split}
 \end{equation}
For consistency with formula \eqref{eq:main16}, we note the following result.
\begin{lemma}
\label{lemmaij}
$\int_Y  \kappa_1  {\partial  N^j_1(x,y)\over \partial y_i}\dy+\int_Y  \kappa_2  {\partial  N^j_2(x,y)\over \partial y_i}\dy 
= \int_Y \kappa_1  {\partial  N^i_1(x,y)\over \partial y_j}\dy+\int_Y  \kappa_2  {\partial  N^i_2(x,y)\over \partial y_j}\dy$ 
\end{lemma}
\begin{proof}
From the cell problem, we have
\begin{equation*}
  \label{eq:2}
 \begin{split}
   \int_Y \kappa_1(e^i+\nabla_y N_1^i)\cdot\nabla_yN_1^j \mathrm{d}y &
   +  \int_Y \kappa_2(e^i+\nabla_y N_2^i)\cdot \nabla_yN_1^j \mathrm{d}y\\
   &+ \int_Y \kappa_1(e^i+\nabla_y N_1^i)\cdot\nabla_yN_2^j\mathrm{d}y
   + \int_Y \kappa_2(e^i+\nabla_y N_2^i)\cdot \nabla_yN_2^j\mathrm{d}y= 0\\
      \int_Y \kappa_1(e^j+\nabla_y N_1^j)\cdot\nabla_yN_1^i\mathrm{d}y 
      &+  \int_Y \kappa_2(e^j+\nabla_y N_2^j)\cdot \nabla_yN_1^i\mathrm{d}y\\
   &+ \int_Y \kappa_1(e^j+\nabla_y N_1^j)\cdot\nabla_yN_2^i\mathrm{d}y
   + \int_Y \kappa_2(e^j+\nabla_y N_2^j)\cdot \nabla_yN_2^i\mathrm{d}y= 0.\\
    \end{split}
\end{equation*}
Thus,
\begin{equation}
  \label{eq:3}
 \begin{split}
   \int_Y \kappa_1\frac{\partial N_1^j}{\partial y_i}\mathrm{d}y 
   +  \int_Y \kappa_2(e^i+\nabla_y N_2^i)\cdot \nabla_yN_1^j\mathrm{d}y
   + \int_Y \kappa_1(e^i+\nabla_y N_1^i)\cdot\nabla_yN_2^j\mathrm{d}y
   + \int_Y \kappa_2 \frac{\partial N_2^j}{\partial y_i}\mathrm{d}y\\
 =   \int_Y \kappa_1\frac{\partial N_1^i}{\partial y_j} \mathrm{d}y
 +  \int_Y \kappa_2(e^j+\nabla_y N_2^j)\cdot \nabla_yN_1^i\mathrm{d}y
   + \int_Y \kappa_1(e^j+\nabla_y N_1^j)\cdot\nabla_yN_2^i \mathrm{d}y
   + \int_Y \kappa_2 \frac{\partial N_2^i}{\partial y_j}\mathrm{d}y   
    \end{split}
\end{equation}
Now we  show
\begin{equation*}
  \label{eq:4}
 \begin{split}
   \int_Y \kappa_2(e^i+\nabla_y N_2^i)\cdot \nabla_yN_1^j \mathrm{d}y
   + \int_Y \kappa_1(e^i+\nabla_y N_1^i)\cdot\nabla_yN_2^j\mathrm{d}y\\
 =    \int_Y \kappa_2(e^j+\nabla_y N_2^j)\cdot \nabla_yN_1^i\mathrm{d}y
   + \int_Y \kappa_1(e^j+\nabla_y N_1^j)\cdot\nabla_yN_2^i\mathrm{d}y.
    \end{split}
\end{equation*}
From the cell problem, we know that
\begin{equation}
  \label{eq:5}
 \begin{split}
   \int_Y \kappa_2(e^i+\nabla_y N_2^i)\cdot \nabla_yN_1^j \mathrm{d}y
   + \int_Y \kappa_1(e^i+\nabla_y N_1^i)\cdot\nabla_yN_2^j \mathrm{d}y\\
   = \int_Y Q(N_1^i - N_2^i)N_1^j + Q(N_2^i-N_1^i)N_2^j \mathrm{d}y\\
   = \int_Y Q(N_1^i N_1^j - N_2^i N_1^j + N_2^i N_2^j - N_1^i N_2^j)\mathrm{d}y\\
   = \int_Y Q(N_1^j-N_2^j)N_1^i + Q(N_2^j-N_1^j)N_2^i\mathrm{d}y\\
   =\int_Y \kappa_2(e^j+\nabla_y N_2^j)\cdot \nabla_yN_1^i \mathrm{d}y
   + \int_Y \kappa_1(e^j+\nabla_y N_1^j)\cdot\nabla_yN_2^i \mathrm{d}y
    \end{split}
\end{equation}
Thus, by (\ref{eq:3}) and (\ref{eq:5}), we have the result.
\end{proof}
\begin{theorem}
Assume that the solution $N_1^i$ and $N_2^i$ of cell problem \eqref{eq:main17} belong to $C^2(\bar\Omega,C^2(\bar Y))$ and the coefficients $\kappa_1$ and $\kappa_2$ belong to $C^1(\bar\Omega,C^1(\bar Y))$.
The limit function $u_0$ of the sequences $u_1^\ep$, $u_2^\ep$ is the unique solution of the homogenized 
equation (\ref{eq:main15}) with the initial condition \eqref{eq:initcond}.
\end{theorem}
\begin{proof}
Note that from the equation (\ref{eq:main67}), we obtain
\begin{equation*}
\label{eq:main91}
 \begin{split}
\int_0^T\int_{\Omega}{\mathcal C}_{11}^\epsilon{\partial u_1^\epsilon  \over \partial t}\phi \mathrm{d}x \psi \mathrm{d}t
+ \int_0^T\int_{\Omega} \kappa_1^\epsilon\nabla u_1^\epsilon \cdot \nabla \phi \mathrm{d}x \psi\mathrm{d}t
+\int_0^T\int_{\Omega}{\mathcal C}_{22}^\epsilon{\partial u_2^\epsilon  \over \partial t} \phi \mathrm{d}x \psi\mathrm{d}t
+ \int_0^T\int_{\Omega} \kappa_2^\epsilon\nabla u_2^\epsilon \cdot \nabla \phi \mathrm{d}x \psi\mathrm{d}t\\
 = 2\int_0^T\int_{\Omega}q \phi \mathrm{d}x \psi \mathrm{d}t.
 \end{split}
 \end{equation*}
for all $\phi\in \mathcal{C}^\infty_0 (\Omega)$ and $\psi\in C^\infty_0((0,T))$.
Passing to the limit, from (\ref{eq:main85}), Lemmas \ref{lemma2} and \ref{lemmaij}, we have
\begin{equation*}
 \label{eq:main93}
 \begin{split}
\int_0^T \int_\Omega \{(\int_Y {\mathcal C}_{11} \mathrm{d}y ) +(\int_Y {\mathcal C}_{22}  \mathrm{d}y )\}
{\partial u_{0} \over \partial t} \phi \mathrm{d}x \psi \mathrm{d}t\\
= \int_0^T \int_\Omega\div(\kappa^*_1 \nabla u_0) \phi \mathrm{d}x \psi \mathrm{d}t
+ \int_0^T \int_\Omega\div(\kappa^*_2 \nabla u_0)\phi \mathrm{d}x \psi \mathrm{d}t
+\int_0^T \int_\Omega 2q \phi \mathrm{d}x \psi \mathrm{d}t
\end{split}
\end{equation*}
where 
\begin{equation*}
 \label{eq:main94}
 \begin{split}
\kappa^*_{1ij}(x) =  \int_Y  \kappa_1 (x,y)(\delta_{ij} + {\partial  N^j_1(x,y)\over \partial y_i}) \dy\\
\kappa^*_{2ij}(x) =  \int_Y  \kappa_2 (x,y)(\delta_{ij} + {\partial  N^j_2(x,y)\over \partial y_i}) \dy.
\end{split}
\end{equation*}
We now show the initial condition.
Adding (\ref{eq:main67}) and (\ref{eq:main68}), we have
\begin{equation*}
{\mathcal C}_{11}^\epsilon{\partial u_1^\epsilon  \over \partial t} + {\mathcal C}_{22}^\epsilon{\partial u_2^\epsilon  \over \partial t}
-\nabla\cdot (\kappa_1^\ep \nabla u_1^\ep) - \nabla\cdot (\kappa_2^\ep \nabla u_2^\ep) = 2q.
\end{equation*}
As $u_1^\ep$ and $u_2^\ep$ are bounded in $L^2(0,T;V)$, we deduce that ${\mathcal C}_{11}^\epsilon{\partial u_1^\epsilon  \over \partial t} + {\mathcal C}_{22}^\epsilon{\partial u_2^\epsilon  \over \partial t}$ is bounded in $L^2(0,T;V')$.
Let $\psi(t,x) \in {\mathcal C}_0^\infty (0,T;V)$, i.e. $\psi(0,x) = \psi(T,x) = 0$. We have
\begin{equation*}
\begin{split}
\int_0^T \int_\Omega \big( {\mathcal C}_{11}^\epsilon{\partial u_1^\epsilon  \over \partial t} + {\mathcal C}_{22}^\epsilon{\partial u_2^\epsilon  \over \partial t} \big) \psi \dx \dt = - \int_0^T \int_\Omega \big( {\mathcal C}_{11}^\epsilon  u_1^\epsilon  + {\mathcal C}_{22}^\epsilon u_2^\epsilon \big) \frac{\partial \psi}{\partial t} \dx \dt \\
\rightarrow - \int_0^T \int_\Omega \big(\langle {\mathcal C}_{11} \rangle+ \langle{\mathcal C}_{22} \rangle\big) u_0 \frac{\partial \psi}{\partial t} \dx \dt = \int_0^T \int_\Omega \big(\langle {\mathcal C}_{11} \rangle+ \langle{\mathcal C}_{22} \rangle\big)  \frac{\partial u_0}{\partial t} \psi \dx \dt.
\end{split}
\end{equation*}
This shows that the weak limit of ${\mathcal C}_{11}^\epsilon{\partial u_1^\epsilon  \over \partial t} + {\mathcal C}_{22}^\epsilon{\partial u_2^\epsilon  \over \partial t} $ in $L^2(0,T;V')$ is $ \big(\langle {\mathcal C}_{11} \rangle+ \langle{\mathcal C}_{22} \rangle\big)  \frac{\partial u_0}{\partial t}$. Now we choose $\psi \in {\mathcal C}^\infty (0,T;V) $ so that $\psi(T,x)=0$. Then

\begin{equation*}
\begin{split}
\int_0^T \int_\Omega \big( {\mathcal C}_{11}^\epsilon{\partial u_1^\epsilon  \over \partial t} + {\mathcal C}_{22}^\epsilon{\partial u_2^\epsilon  \over \partial t} \big) \psi \dx \dt &= - \int_0^T \int_\Omega \big( {\mathcal C}_{11}^\epsilon  u_1^\epsilon  + {\mathcal C}_{22}^\epsilon u_2^\epsilon \big) \frac{\partial \psi}{\partial t} \dx \dt + \int_\Omega \big( {\mathcal C}_{11}^\epsilon  u_1^\epsilon(0,x)  + {\mathcal C}_{22}^\epsilon u_2^\epsilon(0,x) \big)\psi(0,x)\dx \\
&\rightarrow - \int_0^T \int_\Omega \big(\langle {\mathcal C}_{11} \rangle+ \langle{\mathcal C}_{22} \rangle\big) u_0 \frac{\partial \psi}{\partial t} \dx \dt + \int_\Omega \big(\langle {\mathcal C}_{11} \rangle g_1+ \langle{\mathcal C}_{22} \rangle g_2 \big)  \psi(0,x) \dx.
\end{split}
\end{equation*}
On the other hand
\begin{equation*}
\begin{split}
\int_0^T \int_\Omega \big( {\mathcal C}_{11}^\epsilon{\partial u_1^\epsilon  \over \partial t} + {\mathcal C}_{22}^\epsilon{\partial u_2^\epsilon  \over \partial t} \big) \psi \dx \dt 
\rightarrow  \int_0^T \int_\Omega \big(\langle {\mathcal C}_{11} \rangle+ \langle{\mathcal C}_{22} \rangle\big)  \frac{\partial u_0}{\partial t} \psi \dx \dt.\\
=- \int_0^T \int_\Omega \big(\langle {\mathcal C}_{11} \rangle+ \langle{\mathcal C}_{22} \rangle\big) u_0 \frac{\partial \psi}{\partial t} \dx \dt + \int_\Omega \big(\langle {\mathcal C}_{11} \rangle+ \langle{\mathcal C}_{22} \rangle \big) u_0(0,x) \psi(0,x) \dx.
\end{split}
\end{equation*}
This shows that $\big(\langle {\mathcal C}_{11} \rangle+ \langle{\mathcal C}_{22} \rangle \big) u_0(0,x) = \langle {\mathcal C}_{11} \rangle g_1(x)+ \langle{\mathcal C}_{22} \rangle g_2(x)$. i.e. the initial condition of $u_0$ is
\begin{equation}
u_0(0,x) = 
\frac{\langle \mathcal{C}_{11} \rangle g_1(x) + \langle \mathcal{C}_{22} \rangle g_2(x)}{\langle \mathcal{C}_{11} \rangle + \langle \mathcal{C}_{22} \rangle}
\end{equation}
\end{proof}

\section{Conclusions}

In this paper, we developed an efficient algorithm for computing the effective coefficients of a coupled multiscale multi-continuum system. We derived the coupled cell problems and the homogenized equation from twosale asymptotic expansion. We solved the cell problems using hierarchical FE algorithm and use the solutions to compute the effective coefficients. To establish the hierarchical FE algorithm, we first constructed a dense hierarchy of macrogrids and the corresponding nested FE spaces. Based on the hierarchy, we solve the cell problems using different resolution FE spaces at different macroscopic points. We use solutions solved with a higher level of accuracy to correct solutions obtained with a lower level of accuracy at nearby macroscopic points. We rigorously showed that this hierarchical FE method achieves the same order of accuracy as the reference full solve where cell problems at every macroscopic point are solved with the highest level of accuracy, at a significantly reduced computational cost, using an essentially optimal number of degrees of freedom. For numerical example, we applied this algorithm to a multi-continuum model in a two dimensional domain. The algorithm was implemented on  macroscopic points in a one dimensional domain. The numerical results strongly support the error estimates we provided in section 3. 

\bigskip

\textbf{Acknowledgment}
{A part of this work is conducted when Jun Sur Richard Park was a visiting PhD student at Nanyang Technological University (NTU) under East Asia and Pacific Summer Institutes (EAPSI) programme organized by the US National Science Foundation (NSF) and Singapore National Research Foundation (NRF) under Grant No. 1713805. 
Jun Sur Richard Park thanks US NSF and Singapore NRF for the financial support and NTU for hospitality.} Viet Ha Hoang is supported by the MOE AcRF Tier 1 grant RG30/16 and the MOE Tier 2 grant MOE2017-T2-2-144.

\bibliographystyle{siam}
\bibliography{referencesHier}

\end{document}